\documentclass{amsart}
\usepackage{amsmath}
\usepackage{amssymb}
\usepackage{amsxtra}
\usepackage{tikz}
\usepackage{amscd}
\usepackage{amsthm}
\usepackage[all]{xy}
\usepackage{amsfonts}
\usepackage{mathrsfs}
\usepackage{times}
\usepackage{anysize}

\marginsize{1in}{1in}{1in}{1.5in}

\newtheorem{lemma}{Lemma}[section]
\newtheorem{example}[lemma]{Example}
\newtheorem{theorem}[lemma]{Theorem}
\newtheorem{conjecture}[lemma]{Conjecture}
\newtheorem{corollary}[lemma]{Corollary}
\newtheorem{remark}[lemma]{Remark}
\theoremstyle{definition}
\newtheorem{definition}[lemma]{Definition}
 
\newcommand{\jcx}{\overline{JC}_x} 
\newcommand{\Td}{\mathbf{T}_{\delta}}

\DeclareMathOperator{\bounce}{bounce}
 
\author{Evgeny Gorsky\and Mikhail Mazin}
\title{Compactified Jacobians and $q,t$-Catalan numbers, II }

\begin{document}

\begin{abstract}
We continue the study of the rational-slope generalized $q,t$-Catalan numbers $c_{m,n}(q,t)$.
We describe generalizations of the bijective constructions of J. Haglund and N. Loehr and use them
to prove a weak symmetry property $c_{m,n}(q,1)=c_{m,n}(1,q)$ for $m=kn\pm 1$. We give a bijective proof of the full symmetry $c_{m,n}(q,t)=c_{m,n}(t,q)$ for $\min(m,n)\le 3$. As a corollary of these combinatorial constructions, we give a simple formula for the Poincar\'e polynomials of compactified Jacobians of plane curve singularities  $x^{kn\pm 1}=y^n$. We also give a geometric interpretation of a relation between rational-slope Catalan numbers and the theory of $(m,n)$-cores discovered by J. Anderson.

\end{abstract}

\maketitle

\section{Introduction}

Let $m$ and $n$ be two coprime positive integers.

\begin{definition}
Consider a rectangle with width $m$ and height $n$. The rational-slope Catalan number $$c_{m,n}=\frac{(m+n-1)!}{m!n!}$$ is defined (e.g. \cite{arms}) as the number of lattice paths from the northwest corner to the southeast corner of this rectangle, staying below the diagonal connecting these corners. We will denote by $Y_{m,n}$ the set of all such paths, so that $|Y_{m,n}|=c_{m,n}$.
\end{definition}

In the case $m=n+1$ the number $c_{m,n}$ coincides with the usual Catalan number. In \cite{GH} M. Haiman and A. Garsia proposed a bivariate generalization of Catalan numbers whose combinatorial definition (\cite{gahagl}) can be naturally extended to the rational-slope case.

\begin{definition}
Let $D$ be the Young diagram. Following \cite{lowa}, we define the statistics
$$h_{+}^{\frac{n}{m}}(D):=\sharp \left\{c\in D : {\frac{a(c)}{l(c)+1}}\le \frac{n}{m}< \frac{a(c)+1}{l(c)}\right\},$$
and $a(c)$ and $l(c)$ denote the arm- and leg-length for a box $c\in D$.
\end{definition}

\begin{definition}(\cite{GM})
We define the rational-slope $q,t$-Catalan number by the formula
$$c_{m,n}(q,t)=\sum_{D\in Y_{m,n}}q^{\delta-|D|}t^{h_{+}^{\frac{n}{m}}(D)},$$
where $\delta=\frac{(m-1)(n-1)}{2}$.
\end{definition}

Note that the polynomials $c_{m,n}(q,t)$ are special cases of more general polynomials considered in \cite{lowa}. For the reader's convenience, from now on we will write $h_{+}$ instead of $h_{+}^{\frac{n}{m}}.$ 

It has been proved in \cite{gahagl} that $c_{m,n}(q,t)$ coincides  with the $q,t$-Catalan numbers 
of A. Garsia and M. Haiman (\cite{GH}) for $m=n+1$. It has been conjectured in \cite{loehr} that $c_{n,kn+1}(q,t)$ coincides with the $k$-analogue of them. The following two conjectures are motivated by the above coincidences. Both conjectures appeared previously in the literature, but in different setups.

\begin{conjecture}[Symmetry conjecture]\label{SymConj}
The function $c_{m,n}(q,t)$ satisfies the functional equation
$$
c_{m,n}(q,t)=c_{m,n}(t,q).
$$
\end{conjecture}

This conjecture is a special case of the Conjecture $22$ of Loehr and Warrington (\cite{lowa}). The second conjecture is a weaker version of the first one:

\begin{conjecture}[Weak symmetry conjecture]\label{WeakSymConj}
The function $c_{m,n}(q,t)$ satisfies the functional equation
$$
c_{m,n}(q,1)=c_{m,n}(1,q).
$$
\end{conjecture}

Conjecture \ref{SymConj} was proved in \cite{gahagl} for $m=n+1.$ It was deduced by Garsia and Haglund from some nontrivial identities 
for $q,t$-Catalan numbers. For $m=kn+1$ a similar statement was conjectured in \cite{loehr}. No bijective proof in any of these cases is known yet.

\begin{theorem}
The symmetry conjecture holds for $m\le 3$. 
\end{theorem}

In the proof we construct an explicit bijection exchanging the area and $h_{+}$ statistics (see Section \ref{Section m<4}). After the first version of this paper was submitted,  K. Lee, L. Li and N. Loehr obtained an independent proof  of this result and extended it to the case $m=4$ (see \cite{lll}).

The weak symmetry property was proved by an explicit bijective construction for $m=n+1$ by J. Haglund (\cite{hagl}) and for $m=kn+1$ by N. Loehr (\cite{loehr}). Following constructions from \cite{GM}, we introduce two maps $G_{m}:Y_{m,n}\to Y_{m,n}$ and $G_{n}:Y_{m,n}\to Y_{n,m}$.  
It was conjectured in \cite{GM} that $G_m$ is a bijection.

\begin{theorem}
The maps $G_m$ and $G_n$ satisfy the following properties:
\begin{enumerate}
\item $|G_{m}(D)|=|G_{n}(D)|=\delta-h_{+}(D).$
\item There exist an involution $D\rightarrow \widehat{D}$ such that $G_m(D)=\left(G_n(\widehat{D})\right)^T.$ 
\item If $m=n+1$, then $G_{n}=G_{n+1}$. 
\item If $m=kn\pm 1$, then $G_{m}$  is bijective.
\item If $m=kn+1$, then $G_{m}$ coincides with the Haglund-Loehr bijection.
\end{enumerate}
\end{theorem}

Our description of the map $G_m$ is independent of the Haglund-Loehr construction. We define these maps in terms of semimodules over the semigroup generated by $m$ and $n.$ Constructing the inverse of $G_m$ translates into the reconstruction of a semimodule from a certain collection of data. This approach allowed us to prove bijectivity in the case $m= kn-1,$ which is a new result. Using the property (1), one can show that the bijectivity of $G_{m}$ (or, equivalently, $G_n$) is sufficient to prove the weak symmetry conjecture.

\begin{corollary}
\label{WScor}
The weak symmetry conjecture holds for $m=kn\pm 1$.
\end{corollary}


We apply the above combinatorial results to study the geometry of the {\em Jacobi factor}, a certain moduli spaces associated with a plane curve singularity (see \cite{piont},\cite{AIK},\cite{beauville} and Section \ref{sec:jc} for the definitions). Of our particular interest is the Jacobi factor of the plane curve singularity  $\{x^m=y^n\}$,
which can be identified with a certain affine Springer fiber (\cite{LS},\cite{hikita},\cite{GORS}).
The following theorem is the main result of \cite{GM}:

\begin{theorem}(\cite{GM})
Consider a plane curve singularity $\{x^m=y^n\}$. Then its Jacobi factor admits an affine cell decomposition. The cells are parametrized by Young diagrams $D$ contained in $m\times n$ rectangle below the diagonal. The dimension for the cell $C_D$ in the Jacobi factor can be computed in terms of $D$ as follows:
$$
\dim C_{D}=\frac{(m-1)(n-1)}{2}-h_{+}^{\frac{n}{m}}(D).
$$ 
\end{theorem}

It was pointed out in \cite{GM} that the weak symmetry conjecture implies a remarkably simple formula for the Poincar\'e polynomial of the Jacobi factor:

$$
P_{\jcx}(t)=\sum_{D}t^{2\dim C_{D}}=\sum_{D}t^{2\left(\delta-h_{+}(D)\right)}=$$ $$=t^{2\delta}c_{m,n}(1,t^{-2})\stackrel{WS}{=}t^{2\delta}c_{m,n}(t^{-2},1)=\sum_{D}t^{2|D|}.
$$

A similar formula for the Poincar\'e polynomial was conjectured by Lusztig and Smelt at the last page of \cite{LS}.
In the case $m=kn\pm 1$ the Corollary \ref{WScor} implies the following result. 

\begin{corollary}
If $m=kn\pm 1$ then the Poincar\'e polynomial of the Jacobi factor of the plane curve singularity $\{x^m=y^n\}$ has the form
$$P_{\jcx}(t)=\sum_{D}t^{2|D|},$$
where the summation is done over all Young diagrams in $m\times n$ rectangle below the diagonal.
\end{corollary}

Following the ideas of J. Anderson (\cite{anderson}), we prove
the following result. Recall that $\delta=\frac{(m-1)(n-1)}{2}$. There is a canonical embedding of the Jacobi factor into the Grassmannian
$Gr(\delta,V)$, where $V$ is the space of polynomials in one variable of degree less or equal to $2\delta-1$. The Grassmannian $Gr(\delta,V)$ has a cell decomposition by Schubert 
cells enumerated by Young diagrams contained in a $\delta\times \delta$ square. The cell decomposition of $\jcx$ is given by  intersections 
with the Schubert cells (\cite{piont},\cite{GP}).

\begin{theorem}
A cell in $Gr(\delta,V)$  has non-empty intersection with $\jcx$ if
and only if the corresponding Young diagram is a simultaneous $(m,n)$-core, i.e. it does not have hooks of length $m$ or $n$.
\end{theorem}

\section{Homology of the compactified Jacobian}
\label{sec:jc}
\subsection{Cell decomposition}
In our previous paper \cite{GM} we studied the combinatorics of cell decompositions of the Jacobi factors of some plane curve singularities.

Let  $x\in C\subset\mathbb C^2$ be a unibranched plane curve singularity, let $t$ be a normalizing parameter on $C$ at $x,$ and $R\subset\mathbb C[[t]]$ be the complete local ring at $x.$ Let $\delta=\dim (\mathbb C[[t]]/R).$ Since $x\in C$ is a plane curve singularity, it follows that $t^{2\delta}\mathbb C[[t]]\subset R.$ Let $V=\mathbb C[[t]]/t^{2\delta}\mathbb C[[t]].$

\begin{definition}
The {\it Jacobi factor} $\jcx$ is the space of $R$-submodules $M\subset\mathbb C[[t]],$ such that $M\supset t^{2\delta}\mathbb C[[t]]$ and $\dim(\mathbb C[[t]]/M)=\delta.$

In other words, $\jcx$ is isomorphic to the subvariety of the Grassmannian $Gr(\delta,V),$ consisting of subspaces invariant under $R$-action.
\end{definition}

Beauville proved in \cite{beauville} that the compactified Jacobian of a complete rational curve is homeomorphic to the direct product of Jacobi factors of its singularities. To study the topology of the compactified Jacobian of a singular rational curve it is sufficient to study the topology of a single Jacobi factor.

In \cite{piont} Piontkowski showed that for some singularities the Jacobi factor admits an algebraic cell decomposition. Following his work, we gave a combinatorial description of the cell decomposition in the case of a plane curve singularity with one Puiseux pair. In particular, it applies to a curve singularity $x^m=y^n$ for $m$ and $n$ coprime. In this case the Jacobi factor coincides with a certain subvariety of the affine Grassmannian considered first by Lusztig and Smelt in \cite{LS} (see also \cite{hikita}).

Let $$\Gamma=\Gamma_{m,n}=\{am+bn|a,b\ge 0\}\subset \mathbb{Z}_{\ge 0}$$ be the semigroup generated by $m$ and $n$.
A subset $\Delta\subset \mathbb{Z}_{\ge 0}$ is called a semimodule over $\Gamma$, if $\Delta+\Gamma\subset \Delta$.
It is zero-normalized, if $\min(\Delta)=0$. 

A number $a$ is called a $n$-generator of $\Delta$, if $a\in \Delta$ and $a-n\notin \Delta$.
Every $\Gamma$-semimodule has $n$ distinct $n$-generators.

\begin{definition}
Let $\Delta$ be a $\Gamma$-semimodule, and $x\in \mathbb{Z}$. We define
$$g_{n,\Delta}(x):=\sharp \left([x,x+m)\setminus \Delta\right).$$
If a semimodule $\Delta$ is fixed, we write $g_{n}(x)$ instead of $g_{n,\Delta}(x)$ for the reader's convenience.
\end{definition}
 
\begin{theorem}(\cite{piont}) The Jacobi factor for a plane curve singularity with one Puiseux pair $(m,n)$ 
admits an affine cell decomposition. The cells $C_{\Delta}$ are enumerated by the 0-normalized
$\Gamma$--semimodules $\Delta$, and their dimensions can be computed by the formula
\begin{equation}
\label{dimp}
\dim C_{\Delta}=\sum_{j=1}^{n}g_{n,\Delta}(a_j)
\end{equation}
where $a_j$ are the $n$-generators of $\Delta$.
\end{theorem}

It has been remarked in \cite{GM} that this definition has a natural combinatorial interpretation.
Let us label the box $(x,y)$ of the positive quadrant by the number $mn-m(1+x)-n(1+y)$. The corner box $(0,0)$ is labelled by $2\delta-1=mn-m-n$ and the numbers decrease by $m$ in east direction and by $n$ in north direction. One can check that every number from $\mathbb{N}\setminus \Gamma$  appears exactly once 
in the $m\times n$ rectangle below the diagonal (it follows e.g. from Lemma \ref{gammasym} below).

Given a $\Gamma$--semimodule $\Delta$, let us mark the boxes labelled by numbers from $\Delta\setminus \Gamma$ and denote the resulting set of boxes by $D(\Delta)$.

\begin{theorem}(\cite{GM})
For a $\Gamma$--semimodule $\Delta$ the set $D(\Delta)$ is a Young diagram. The correspondence $D$ between semimodules and Young diagrams below the diagonal is bijective.
The dimension of a cell $C_{\Delta}$ in the Jacobi factor can be computed in terms of $D(\Delta)$ as follows:
\begin{equation}
\label{dimgm}
\dim C_{\Delta}=\frac{(m-1)(n-1)}{2}-h_{+}(D(\Delta)).
\end{equation}
\end{theorem}

\begin{example}
Consider the semigroup $\Gamma$ generated by the numbers 5 and 7, and a $\Gamma$-semimodule
$\Delta=\mathbb{Z}_{\ge 0}\backslash \{1,2,3,4,6,9\}=\{0,5,7,8,10,11,\dots\}.$
The diagram $D(\Delta)$ is shown in Figure \ref{5times7}. The $5$--generators of $\Delta$ are $\{a_0,a_1,a_2,a_3,a_4\}=\{0,7,8,11,14\},$
and $$g_5(0)=\sharp([0,7)\setminus \Delta)=5,\ g_5(7)=\sharp([7,14)\setminus \Delta)=1,\ g_5(8)=\sharp([8,15)\setminus \Delta)=1,$$
and $g_5(11)=g_5(14)=0$. Therefore $$\dim C_{\Delta}=g_5(0)+g_5(7)+g_5(8)+g_5(11)+g_5(14)=7.$$

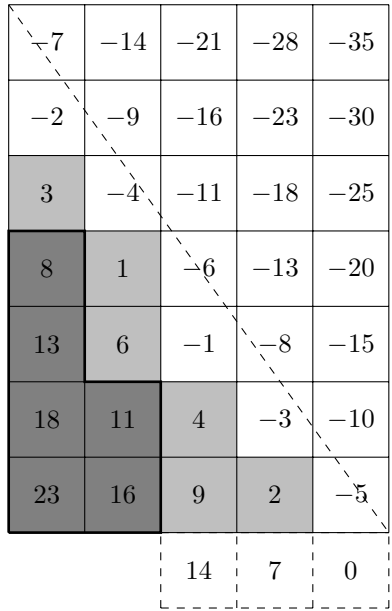
\begin{figure}[ht]
\begin{center}
\begin{tikzpicture}
\fill [color=lightgray] (0,0)--(0,5)--(1,5)--(1,4)--(2,4)--(2,2)--(3,2)--(3,1)--(4,1)--(4,0)--(0,0);
\fill [color=gray] (0,0)--(0,4)--(1,4)--(1,2)--(2,2)--(2,0)--(0,0);
\draw (0,0) grid (5,7);
\draw [dashed] (2,-1) grid (5,0);
\draw (2.5,-0.5) node {$14$};
\draw (3.5,-0.5) node {$7$};
\draw (4.5,-0.5) node {$0$};
\draw (0.5,0.5) node {$23$};
\draw (1.5,0.5) node {$16$};
\draw (2.5,0.5) node {$9$};
\draw (3.5,0.5) node {$2$};
\draw (4.5,0.5) node {$-5$};
\draw (0.5,1.5) node {$18$};
\draw (1.5,1.5) node {$11$};
\draw (2.5,1.5) node {$4$};
\draw (3.5,1.5) node {$-3$};
\draw (4.5,1.5) node {$-10$};
\draw (0.5,2.5) node {$13$};
\draw (1.5,2.5) node {$6$};
\draw (2.5,2.5) node {$-1$};
\draw (3.5,2.5) node {$-8$};
\draw (4.5,2.5) node {$-15$};
\draw (0.5,3.5) node {$8$};
\draw (1.5,3.5) node {$1$};
\draw (2.5,3.5) node {$-6$};
\draw (3.5,3.5) node {$-13$};
\draw (4.5,3.5) node {$-20$};
\draw (0.5,4.5) node {$3$};
\draw (1.5,4.5) node {$-4$};
\draw (2.5,4.5) node {$-11$};
\draw (3.5,4.5) node {$-18$};
\draw (4.5,4.5) node {$-25$};
\draw (0.5,5.5) node {$-2$};
\draw (1.5,5.5) node {$-9$};
\draw (2.5,5.5) node {$-16$};
\draw (3.5,5.5) node {$-23$};
\draw (4.5,5.5) node {$-30$};
\draw (0.5,6.5) node {$-7$};
\draw (1.5,6.5) node {$-14$};
\draw (2.5,6.5) node {$-21$};
\draw (3.5,6.5) node {$-28$};
\draw (4.5,6.5) node {$-35$};
\draw [dashed] (0,7)--(5,0);
\draw [line width=1] (0,0)--(0,4)--(1,4)--(1,2)--(2,2)--(2,0)--(0,0);
\end{tikzpicture}
\caption{The diagram $D(\Delta)$ for $n=5,\ m=7,$ and $\Delta=\mathbb{Z}_{\ge 0}\backslash \{1,2,3,4,6,9\}=\{0,5,7,8,10,11,\dots\}.$}\label{5times7}
\end{center}
\end{figure}
\end{example}

\subsection{Duality}

Let us recall a useful symmetry property for the plane curve semigroup $\Gamma$.

\begin{lemma}(e.g. \cite{kunz},\cite{delgado})
\label{gammasym}
Let $\Gamma$ be the semigroup of a unibranched plane curve singularity and let $\delta$ be its $\delta$-invariant. Then
$$a\in \Gamma\quad \Longleftrightarrow 2\delta-1-a\notin \Gamma.$$
\end{lemma}

For every $\Gamma$--semimodule $\Delta$ we can  consider the {\em dual semimodule} 
$$
\Delta^{*}:=\{\phi\ :\ \phi+\Delta\subset \Gamma\}.
$$

\begin{lemma}
\label{mstar}
The dual semimodule $\Delta^{*}$ can be characterised by the equation
$$\Delta^{*}=(2\delta-1)-(\mathbb{Z}\setminus \Delta).$$
\end{lemma}

\begin{proof}
Note that $(2\delta-1)-(\mathbb{Z}\setminus \Delta)$ is a $\Gamma$--semimodule.
Indeed, if $x\in \mathbb{Z}\setminus\Delta$ then $x-m,x-n\in \mathbb{Z}\setminus\Delta$.

It follows from Lemma \ref{gammasym} that
$$\phi+\Delta\subset \Gamma\Leftrightarrow(2\delta-1)\notin \phi+\Delta.$$
Indeed, if $(2\delta-1)\in \phi+\Delta$ then $\phi+\Delta$ is not a subset of $\Gamma$.
Conversely, if $\exists x\in (\phi+\Delta)\setminus \Gamma$, then
by Lemma \ref{gammasym} $(2\delta-1-x)\in \Gamma$, so $(2\delta-1)\in x+\Gamma\subset \phi+\Delta$.
Therefore 
$$
\phi+\Delta\subset \Gamma\Leftrightarrow(2\delta-1)\notin \phi+\Delta \Leftrightarrow
$$
$$
\Leftrightarrow(2\delta-1-\phi)\notin \Delta  \Leftrightarrow \phi\in (2\delta-1)-(\mathbb{Z}\setminus \Delta).
$$
\end{proof}

\begin{remark}
Lemmas \ref{mstar} and \ref{gammasym} are combinatorial analogues of the Gorenstein property for the plane curve singularities (e.g. \cite{kunz},\cite{delgado}).
\end{remark}

We normalize $\Delta^{*}$ so that it starts from $0$ and denote the normalized semimodule by $\widehat{\Delta}$.
Following Lemma \ref{mstar},
we have
\begin{equation}
\label{DeltaHat}
\widehat{\Delta}=\max (\mathbb{Z}\setminus \Delta)-(\mathbb{Z}\setminus \Delta).
\end{equation}

\begin{lemma}
\label{hathat}
One has $\widehat{(\widehat{\Delta})}=\Delta$.
\end{lemma}

\begin{proof}
It follows from (\ref{DeltaHat}) that 
$$\mathbb{Z}\setminus\widehat{\Delta}=\max (\mathbb{Z}\setminus \Delta)-\Delta,$$
and therefore
$$\max(\mathbb{Z}\setminus\widehat{\Delta})=\max (\mathbb{Z}\setminus \Delta).$$
\end{proof}

\begin{example}
We list the combinatorial types of $\Gamma$--semimodules $\Delta$, $\Delta^{*}$ and $\widehat{\Delta}$ for $(m,n)=(3,4)$ in Table \ref{tab1}.

\begin{table}
\label{tab1}
\centering
\begin{tabular}[c]{|c|c|c|}
\hline
$\Delta$ & $\Delta^{*}$ & $\widehat{\Delta}$\\
\hline
$0,3,4,6,\ldots$ & $0,3,4,6,\ldots$ & $0,3,4,6,\ldots$ \\
$0,3,4,5,6,\ldots$ & $3,4,6,\ldots$ & $0,1,3,4,5,6,\ldots$\\
$0,2,3,4,5,6,\ldots$ & $4,6,\ldots$ & $0,2,3,4,5,6,\ldots$\\
$0,1,3,4,5,6,\ldots$ & $3,6,\ldots$ & $0,3,4,5,6,\ldots$\\
$0,1,2,3,4,5,6,\ldots$ & $6,\ldots$ & $0,1,2,3,4,6,\ldots$\\
\hline
\end{tabular}
\caption{Semimodules $\Delta^{*}$ and $\widehat{\Delta}$ for $(m,n)=(3,4)$}
\end{table}

\end{example}







 
\subsection{Maps $G_{m}$ and $G_{n}$}
\label{section G_m G_n}

Let us return to the formula (\ref{dimp}). Given a $\Gamma$-semimodule $\Delta$, we can consider its $m$-generators $a_1,\ldots,a_m$
and compute 
$$g_{m}(a_i)=\sharp \left([a_i,a_i+n)\setminus \Delta\right).$$

\begin{theorem}(\cite{GM})
The numbers $g_{m}(a_i)$ are decreasing. 
The Young diagram with columns $g_{m}(a_i)$ can be embedded in an $m\times n$ rectangle below the diagonal.
\end{theorem}

This result allows us to consider the map $G_{m}$ from the set of diagrams below the diagonal to itself,
sending $D(\Delta)$ to the diagram  with columns $g_{m}(a_i)$. We will also  use the notation $G_{m}(\Delta)=G_{m}(D(\Delta))$ for a $\Gamma$--semimodule $\Delta$. In a similar way, one can define the map $G_{n}$ by choosing $n$-generators instead of $m$-generators. Note that the diagram $G_n(\Delta)$ has $n$ columns, and is embedded in an $n\times m$ rectangle below the diagonal. 
\begin{theorem}
\label{commdiag}
One has
$$G_{n}(\Delta)=\left(G_{m}(\widehat{\Delta})\right)^{T}.$$
\end{theorem} 

\begin{proof}
Let us recall the notion of an {\em $m$-cogenerator} for the $\Gamma$--semimodule $\Delta$ introduced in \cite{GM}: we call $b$ an $m$-cogenerator for $\Delta$ if $b\notin \Delta$, but $b+m\in \Delta$. Remark that by Lemma \ref{mstar} the $m$-cogenerators of $\Delta$ are in 1-to-1 correspondence with the $m$-generators of $\widehat{\Delta}$. More precisely, if $b$ is an $m$-generator of $\widehat{\Delta}$, then $(\max(\mathbb{Z}\setminus\Delta)-b)$ is an $m$-cogenerator for $\Delta$.

Let $a_1,\ldots,a_n$ be the $n$-generators of $\Delta$ and let $b_1,\ldots, b_m$ be the $m$ --
generators of $\widehat{\Delta}$. One can check that $g_{n}(a_i)$ equals  the number of $m$-cogenerators of $\Delta$ greater than $a_i$:
$$
g_{n}(a_i)=\sharp \left([a_i,a_i+m)\setminus \Delta\right)=\sharp \{j| \max(\mathbb{Z}\setminus\Delta)-b_j>a_i\}=\sharp \{j| a_i+b_j<\max(\mathbb{Z}\setminus\Delta)\}.
$$
Analogously by Lemma \ref{hathat}
$$g_{m}(b_j)=\sharp \{i| a_i+b_j<\max(\mathbb{Z}\setminus\Delta)\}.$$
It is clear that the corresponding Young diagrams are transposed to each other.
\end{proof}

\begin{corollary}
The map $G_{n}$ is bijective if and only if the map $G_{m}$ is bijective.
\end{corollary}

\begin{corollary}
The dimensions of cells in $\jcx$ labelled by $\Delta$ and $\widehat{\Delta}$ are the same:
$$\dim C_{\Delta}=|G_{n}(\Delta)|=|G_{m}(\widehat{\Delta})|=\dim C_{\widehat{\Delta}}.$$
\end{corollary}

\begin{theorem}
\label{nnp1}
If $m=n+1$ then $G_{n+1}=G_{n}$.
\end{theorem}

\begin{proof}
Let $\Delta$ be a $\Gamma$--semimodule and $x\in \Delta$. Let 
$a_{+}(x)$ be the minimal $n$-generator of $\Delta$ greater than or equal to $x$. We set $a_{+}(x)=\infty$ if there are no $n$-generators greater than $x.$ Similarly, let $a_{-}(x)$ be the maximal $(n+1)$-generator of $\Delta$ less than or equal to $x$. Note that since $0$ is an $n$-generator and $x\in\Delta\subset \mathbb Z_{\ge 0},\ a_{-}(x)$ is always well defined.

If $x$ is not an $n$-generator then $x-n\in \Delta$, so $(x-n)+(n+1)=x+1\in \Delta$.
If $x+1$ is not an $n$-generator then $x+2\in \Delta$ etc. By continuing this process
we conclude that $[x,a_{+}(x)]\subset \Delta$. Since $a_{+}(x)-n\notin \Delta$, either $a_{+}(x)+1\notin \Delta$ or
$a_{+}(x)+1$ is a $(n+1)$-generator of $\Delta$.

Analogously, $[a_{-}(x),x]\subset \Delta$ and either $a_{-}(x)-1\notin \Delta$ or
$a_{-}(x)-1$ is an $n$-generator of $\Delta$. Therefore $n$- and $n+1$-generators of $\Delta$ are split into pairs $(a_{-},a_{+})$ such that $[a_{-},a_{+}]\subset \Delta$,
and this is a 1-to-1 correspondence except for the largest $(n+1)$-generator. Since $[a_{-},a_{+}]\subset \Delta$, we have $[a_{-}+n,a_{+}+n+1]\subset \Delta$ and 
$$g_{n+1}(a_{-})=\sharp([a_{-},a_{-}+n)\setminus\Delta)=\sharp([a_{+},a_{+}+n+1)\setminus\Delta)=g_{n}(a_{+}).$$
\end{proof}

\begin{corollary}
If $m=n+1$, then 
$$G_{n}(\widehat{\Delta})=(G_n(\Delta))^{T}.$$
\end{corollary}

\begin{proof}
By Theorems \ref{commdiag} and \ref{nnp1}
$$G_{n}(\widehat{\Delta})=(G_{n+1}(\Delta))^{T}=(G_n(\Delta))^{T}.$$
\end{proof}

\begin{example}
Let us present an example where $G_{n}$ and $G_{m}$ are essentially different.
Let $(m,n)=(3,7)$ and $\Delta=\{0,1,3,4,6,7,8,9,10,11,12,\ldots\}$.
Then the $3$-generators are equal to $0,1,8$, and
$$g_{3}(0)=2,g_{3}(1)=2,g_{3}(8)=0.$$
The $7$-generators are $0,1,3,4,6,9,12$, and
$$g_{7}(0)=g_{7}(1)=g_{7}(3)=g_{7}(4)=1,g_{7}(6)=g_{7}(9)=g_{7}(12)=0.$$

The dual semimodule is 
$\widehat{\Delta}=\{0,3,6,7,8,\ldots\}.$
Its $3$-generators are $0,7,8$, and
$$g_{3}(0)=4,g_{3}(7)=g_{3}(8)=0.$$
The $7$-generators are $0,3,6,8,9,11,12$, and
$$g_{7}(0)=g_{7}(3)=2, g_{7}(6)=g_{7}(8)=g_{7}(9)=g_{7}(11)=g_{7}(12)=0.$$

We illustrate this in Figure \ref{GmGn}.

\begin{figure}
\centering
\begin{tikzpicture}
\draw (-0.5,3) node {$\Delta$};
\draw (-0.5,1) node {$\widehat{\Delta}$};

\draw (0,2.5)--(0,3.5)--(1,3.5)--(1,2.5)--(0,2.5);
\draw (0,3)--(1,3);
\draw (0.5,2.5)--(0.5,3.5);
\draw (0,0)--(0,0.5)--(1,0.5)--(1,0)--(0,0);
\draw (0.5,0)--(0.5,0.5);

\draw (0,0.5)--(0,1.5)--(3.5,1.5)--(3.5,0)--(1,0);
\draw [dashed] (0,1.5)--(3.5,0);

\draw (0,3.5)--(0,4)--(3.5,4)--(3.5,2.5)--(1,2.5);
\draw [dashed] (0,4)--(3.5,2.5);

\draw (0.2,0.2) node {11};
\draw (0.2,0.7) node {4};
\draw (0.7,0.2) node {8};
\draw (0.7,0.7) node {1};
\draw (1.2,0.2) node {5};
\draw (1.7,0.2) node {2};

\draw (0.2,2.7) node {11};
\draw (0.2,3.2) node {4};
\draw (0.7,2.7) node {8};
\draw (0.7,3.2) node {1};
\draw (1.2,2.7) node {5};
\draw (1.7,2.7) node {2};

\draw (2,4.5) node {$D(\cdot)$};

\draw (4.5,0)--(4.5,2)--(5,2)--(5,0)--(4.5,0);
\draw (4.5,0.5)--(5,0.5);
\draw (4.5,1)--(5,1);
\draw (4.5,1.5)--(5,1.5);

\draw (4.5,2.5)--(4.5,3.5)--(5.5,3.5)--(5.5,2.5)--(4.5,2.5);
\draw (5,2.5)--(5,3.5);
\draw (4.5,3)--(5.5,3);

\draw (5,4.5) node {$G_{3}(\cdot)$};

\draw (7,0)--(7,1)--(8,1)--(8,0)--(7,0);
\draw (7,0.5)--(8,0.5);
\draw (7.5,0)--(7.5,1);

\draw (6.5,2.5)--(6.5,3)--(8.5,3)--(8.5,2.5)--(6.5,2.5);
\draw (7,2.5)--(7,3);
\draw (7.5,2.5)--(7.5,3);
\draw (8,2.5)--(8,3);

\draw (7.5,4.5) node {$G_{7}(\cdot)$};

\end{tikzpicture}
\caption{Maps $G_m$ and $G_n$}
\label{GmGn}
\end{figure}
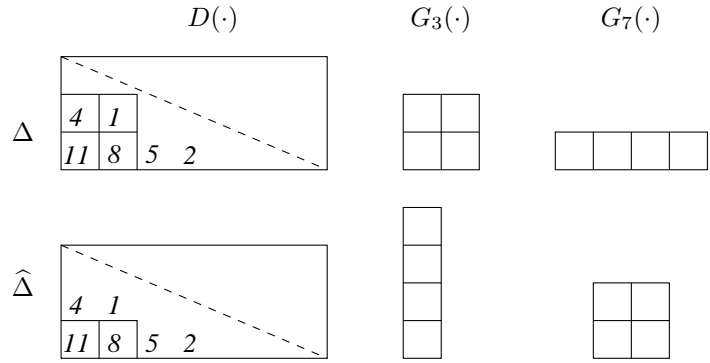
\end{example}

\subsection{$(m,n)$--cores}

\begin{definition}
A Young diagram is called a $p$--core if it does not have boxes with hook length equal to $p$.
\end{definition}

The $p$-core partitions play an important role in the study of representations of symmetric groups over finite fields (see \cite{anderson},\cite{FMS},\cite{puchta} and references therein). J. Anderson observed that the number of partitions that are simultaneously $m$- and $n$-cores is finite:

\begin{theorem}(\cite{anderson})\label{mn-core from semimod}
\label{and}
The number of partitions that are simultaneously $m$- and $n$-cores equals 
$$\frac{1}{m+n}\binom{m+n}{n}.$$
\end{theorem}

The proof in \cite{anderson} uses explicit bijection between $(m,n)$-cores and lattice paths in $m\times n$ rectangle below the diagonal.
Following the analogy between such paths and $\Gamma_{m,n}$--semimodules, we would like to present Anderson's bijection in slightly different form.

\begin{proof}
Given a set $0\in \Delta\subset \mathbb{Z}_{\ge 0}$, we construct the partition $P(\Delta)$ by the following rule. We start from $0\in \Delta$ and read all consecutive integers. If $x\in \Delta$, we move north by 1, if $x\notin \Delta$, we move west by 1. The resulting lattice path bounds from above the Young diagram of the partition $P(\Delta)$.

The partition $P(\Delta)$ has a hook of length $m$ if and only if there are integers $x\in \Delta$ and $y\notin \Delta$ such that $y=x+m$. Therefore $P(\Delta)$ is a simultaneous $(m,n)$-core if and only if $\Delta$ is a $\Gamma_{m,n}$--semimodule. 
\end{proof}

\begin{example}
The $(3,4)$-core corresponding to $\Delta=\{0,3,4,6,\ldots\}$ is shown in Figure \ref{34core}. 

\begin{figure}
\centering
\begin{tikzpicture}
\draw [line width=0.5mm] (1.5,0)--(0.5,0)--(0.5,0.5)--(-0.5,0.5)--(-0.5,1.5)--(-1,1.5)--(-1,2.5);
\draw (0.5,0)--(-1,0)--(-1,1.5);
\draw (0,0)--(0,0.5);
\draw (-0.5,0)--(-0.5,0.5)--(-1,0.5);
\draw (-0.5,1)--(-1,1);
\draw (3,1) node {0};
\draw (3.4,1) node {1};
\draw (3.8,1) node {2};
\draw (4.2,1) node {3};
\draw (4.6,1) node {4};
\draw (5.0,1) node {5};
\draw (5.4,1) node {6};
\draw (5.8,1) node {7};

\draw (3,0.5) node {N};
\draw (3.4,0.5) node {W};
\draw (3.8,0.5) node {W};
\draw (4.2,0.5) node {N};
\draw (4.6,0.5) node {N};
\draw (5.0,0.5) node {W};
\draw (5.4,0.5) node {N};
\draw (5.8,0.5) node {N};

\end{tikzpicture}
\caption{$(3,4)$ core}
\label{34core}
\end{figure}
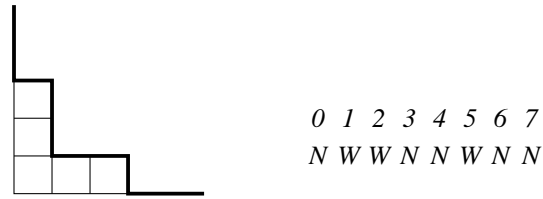
\end{example}
 
\begin{lemma}
\label{conjcore}
The $(m,n)$-core corresponding to $\widehat{\Delta}$ is conjugate to the $(m,n)$-core for $\Delta$.
\end{lemma}

\begin{proof}
By the equation (\ref{DeltaHat}), $\widehat{\Delta}=\max(\mathbb{Z}\setminus \Delta)-(\mathbb{Z}\setminus \Delta)$. Therefore to $x\in \Delta$ we associate an N step in the boundary of the $(m,n)$-core
corresponding to $\Delta$ and a W step in the boundary of the $(m,n)$-core
corresponding to $\widehat{\Delta}$. Analogously,  to $x\not\in \Delta$ we associate a W step in the boundary of the $(m,n)$-core
corresponding to $\Delta$ and an N step in the boundary of the $(m,n)$-core
corresponding to $\widehat{\Delta}$. 
\end{proof}

\begin{theorem}
The number of self-dual $\Gamma$--semimodules equals 
$$\binom{\lfloor \frac{m}{2}\rfloor+\lfloor \frac{n}{2}\rfloor}{ \lfloor \frac{m}{2}\rfloor}.$$
\end{theorem}

\begin{proof}
By Lemma \ref{conjcore} the $\Gamma$--semimodules such that $\Delta=\widehat{\Delta}$ correspond to the self-conjugate $(m,n)$-cores.
The number of such cores was computed by B. Ford, H. Mai and L. Sze in \cite{FMS}, and it is equal to
$\binom{\lfloor \frac{m}{2}\rfloor+\lfloor \frac{n}{2}\rfloor}{ \lfloor \frac{m}{2}\rfloor}.$
\end{proof}

\begin{remark}(\cite{arms}) The number of self-dual modules equals 
$$\binom{\lfloor \frac{m}{2}\rfloor+\lfloor \frac{n}{2}\rfloor}{ \lfloor \frac{m}{2}\rfloor}=\frac{[(m+n-1)!]_{q}}{[m!]_{q}[n!]_{q}}|_{q=-1}.$$
\end{remark}

Let us give a geometric interpretation of the construction from Theorem \ref{and}. 
As it was discussed in the Introduction, the vector space $V=\mathbb{C}[[t]]/t^{2\delta}\mathbb{C}[[t]]$ comes with a natural filtration
$$
V=V_{2\delta}\supset V_{2\delta-1}\supset\ldots \supset V_{0}=0,\quad V_{i}:=t^{2\delta-i}V.
$$
Therefore $Gr(\delta,V)$ has a natural Schubert cell decomposition (see e.g. \cite[p. 74-75]{mist}), which can be described as follows. 

Let $P$ be a Young diagram contained in a $\delta\times\delta$ square. Let $p_1\le\dots\le p_{\delta}$ be its rows. The Schubert cell $C_P\subset Gr(\delta,V)$ consists of subspaces $W\subset V$ such that
$$
\dim(W\cap V_i)=\sharp \{j:p_j+j\le i\}.
$$

Equivalently, one can assign the diagram $P(W)$ to a subspace $W\subset V$ in the following way. Define the subset $\Delta(W)\subset\mathbb Z_{\ge 0}$ as follows:
$$
\Delta(W)=\{d\in\mathbb Z:\exists p\in W, p\in V_{2\delta-d}\setminus V_{2\delta-d-1}\}\cup [t^{2\delta},\infty).
$$
Now let us construct the diagram $P(W)$ from the subset $\Delta(W)$ as in the Theorem \ref{mn-core from semimod}. One can check that $W$ belongs to the Schubert cell $C_{P(W)}\subset Gr(\delta, V).$

Note that according to the construction, $P(W)$ is a simultaneous $m,n$-core iff $\Delta(W)$ is a $\Gamma_{m,n}$--semimodule. It follows immediately that if $W\in\jcx$ then $\Delta(W)$ is a $\Gamma_{m,n}$--semimodule. On the other hand, J. Piontkowski showed in \cite{piont} that for a fixed $\Gamma_{m,n}$--semimodule $\Delta$ there always exist $W\in\jcx$ such that $\Delta(W)=\Delta.$ More precisely, modules $W\in\jcx$ with a fixed semimodule $\Delta(W)$ form a cell in the Piontkowski's cell decomposition of $\jcx.$

Therefore, one gets the following

\begin{theorem}\label{grcore}
Let $P$ be a Young diagram contained in a $\delta\times\delta$ square. Let $C_P\subset Gr(\delta, V)$ be the corresponding Schubert cell. Then the intersection $C_P\cap\jcx$ is non-empty iff $P$ is a simultaneous $m,n$-core, in which case it is the corresponding Piontkowski's affine cell. 
\end{theorem}

\section{Bounce statistics and Poincar\'e polynomials}

Let $\Delta$ be a $\Gamma_{m,n}$--semimodule. Let us recall the following definition:
\begin{definition}
A number $a\in \Delta$ is called an $m$-generator if $a-m\notin\Delta.$ A number $a\notin \Delta$ is called an $m$-cogenerator if $a+m\in\Delta.$ 
\end{definition}
In this section we construct the inverse maps $G_m^{-1}$ in the cases $m=kn\pm 1.$ In other words, we reconstruct a $\Gamma_{m,n}$--semimodule $\Delta$ from the collection of integers $g_m(a_0), g_m(a_1),\dots, g_m(a_{m-1}),$ where $\{0=a_0<\dots<a_{m-1}\}$ are the $m$-generators of $\Delta$. Recall that the function $g_m(x)$ is defined by the formula
$$
g_{m}(x)=\sharp \left([x,x+n)\setminus \Delta\right),
$$
To simplify notations, we will use $g(x)$ instead of $g_m(x)$ in this Section.

In Sections \ref{SS kn+1} and \ref{SS kn-1} we discuss the cases $m=kn+1$ 
and $m=kn-1$ respectively. The logic in both cases is quite similar: first we reconstruct the bounce tree (see Definitions \ref{def tree} and \ref{def tree-1}), and  then use it to reconstruct the semimodule $\Delta$.

In Section \ref{SS bounce} we compare this procedure and the bounce tree with the constructions of Haglund and Loehr.

\subsection{The case $m=kn+1.$}\label{SS kn+1}

 We start with some general facts on generators and cogenerators.

\begin{lemma}\label{arithmetic progression}
Suppose that $x\in \Delta$ is not an $m$-generator. Then $x+n$ is not an $m$-generator as well.
\end{lemma}

\begin{proof}
Indeed, if $x\in\Delta$ is not an $m$-generator, then $x-m\in\Delta.$ But then $x-m+n=x+n-m\in\Delta.$ Therefore, $x+n$ is not an $m$-generator.
\end{proof}


\begin{lemma}
\label{numgens}
Fix an integer $x\in\mathbb Z.$ We have
\begin{enumerate}
\item  The number of $m$-generators in $[x,x+n)$ equals  $g(x-m)-g(x).$
\item  The number of $m$-cogenerators in $[x,x+n)$ equals  $g(x)-g(x+m).$
\item The number of $m$-generators in the interval $[x,x+n)$ is equal to the number of $n$-cogenerators in the interval $[x-m,x).$
\end{enumerate}
\end{lemma}

\begin{proof}
\begin{enumerate}
\item For any integer $y\in [x-m,x-m+n)\backslash\Delta$ the number $y+m$ is either an $m$-generator in $[x,x+m)$ or it is in $[x,x+m)\backslash\Delta.$ Also, for any $z\in [x,x+m)\backslash\Delta$ one has $z-m\in [x-m,x-m+n)\backslash\Delta.$ Therefore, the number of $m$-generators in $[x,x+m)$ is equal to the number of elements in $[x-m,x-m+n)\backslash\Delta$ minus number of elements in $[x,x+n)\backslash\Delta,$ which is $g(x-m)-g(x).$
\item Similar to the previous part: an $m$-cogenerator in $[x,x+n)$ is an integer $y\in [x,x+n)\backslash\Delta,$ such that $y+m\notin [x+m,x+m+n)\backslash\Delta.$
\item Follows immediately from the first two parts.
\end{enumerate}
\end{proof}

Let $m=kn+1$ from now until the end of the Section \ref{SS kn+1}. Our goal is to reconstruct the semimodule $\Delta$ from the numbers $g(a_0),\dots,g(a_{kn}),$ where $a_0,\ldots a_{kn}$ are the $(kn+1)$-generators of $\Delta$ listed in the increasing order. 

\begin{lemma}\label{first in a sequence is a generator lemma}
Suppose that $x\in\Delta$ and $x-\alpha n-1\notin\Delta$ for some $\alpha\in\{0,\dots,k\}$. Then $x$ is a $(kn+1)$-generator.
\end{lemma}

\begin{proof}
Indeed, if $x$ is not a $(kn+1)$-generator, than $x-kn-1\in\Delta.$ But then $(x-kn-1)+(k-\alpha)n=x-\alpha n-1\in\Delta.$ Contradiction. 
\end{proof}

\begin{definition}
For $x\in \Delta$ we define $a_{-}(x)$ to be the maximal $(kn+1)$-generator less than or equal to $x$.
\end{definition}

\begin{corollary}\label{first in a sequence is a generator}
For any $x\in \Delta$ one has $[a_{-}(x),x]\subset \Delta$.
\end{corollary}

\begin{proof}
Follows from Lemma \ref{first in a sequence is a generator lemma} with $\alpha=0.$
\end{proof}

\begin{lemma}\label{gaps}
Consider a $(kn+1)$-generator $a_i.$ Then for any $l>0,$ the interval $J_l:=[a_i-l(kn+1),a_i-lkn-1]$ has empty intersection with $\Delta.$
\end{lemma}

\begin{proof}
The proof is by induction on $l.$ The case $l=1$ is clear. Then observe that 
$$
J_{l+1}=\{a_i-(l+1)(kn+1)\}\cup (J_l-kn). 
$$ 
\end{proof}

We illustrate Lemma \ref{gaps} with the Figure \ref{gapspic}.

\begin{figure}
\centering
\begin{tikzpicture}[scale=0.5]
\draw (0,0) grid +(21,1);

\draw (0.5,0.5) node {$\circ$};
\draw (1.5,0.5) node {$\circ$};
\draw (4.5,0.5) node {$\circ$};
\draw (7.5,0.5) node {$\circ$};
\draw (10.5,0.5) node {$\circ$};
\draw (20.5,0.5) node {$a_i$};

\draw (2,0) .. controls (3,-0.6) and (4,-0.6) .. (5,0);
\draw (5,0) .. controls (6,-0.6) and (7,-0.6) .. (8,0);
\draw (8,0) .. controls (9,-0.6) and (10,-0.6) .. (11,0);
\draw (11,0) .. controls (12,-0.6) and (13,-0.6) .. (14,0);
\draw (14,0) .. controls (15,-0.6) and (16,-0.6) .. (17,0);
\draw (17,0) .. controls (18,-0.6) and (19,-0.6) .. (20,0);

\draw (0,1) .. controls (3,1.9) and (7,1.9) .. (10,1);
\draw (10,1) .. controls (13,1.9) and (17,1.9) .. (20,1);

\end{tikzpicture}
\caption{In this picture $n=3,$ $k=3,$ $a_i$ is a $10$-generator, and "$\circ$"\ indicates an integer not in $\Delta.$ Note that the interval $[a_i-2(kn+1),a_i-2kn-1)]=[a_i-20,a_i-19]$ has empty intersection with $\Delta.$}
\label{gapspic}
\end{figure}

\begin{definition}
We introduce the following notations:
$$
N_{ij}=\left\lceil\frac{a_j-a_i}{n} \right\rceil\mbox{ and }
K_{ij}=\left\lceil\frac{a_j-a_i}{kn+1} \right\rceil.
$$
\end{definition}

\begin{corollary}\label{KtoN}
One has the following formula:
$
K_{ij}=\left\lceil\frac{N_{ij}}{k} \right\rceil.
$
\end{corollary}

\begin{proof}
Note that 
$$
\left\lceil\frac{N_{ij}}{k} \right\rceil=\left\lceil\frac{\left\lceil\frac{a_j-a_i}{n} \right\rceil}{k} \right\rceil=\left\lceil\frac{a_j-a_i}{kn} \right\rceil.
$$ 
Suppose that $K_{ij}\neq\left\lceil\frac{N_{ij}}{k} \right\rceil.$ Then there exists $l,$ such that $a_j-a_i\le l(kn+1)$ and $a_j-a_i>lkn,$ which is equivalent to $a_i\in [a_j-l(kn+1),a_j-lkn).$ Therefore, by Lemma \ref{gaps}, $a_i\notin \Delta.$ Contradiction.
\end{proof}

There are two steps in the reconstruction of $\Delta$ . First we reconstruct the bounce tree $T_{\Delta},$ defined below, and then we use it to recover the $(kn+1)$-generators $a_0,\dots,a_{kn}.$

\begin{definition}
\label{def tree}
The oriented graph $T_{\Delta}$ is defined as follows. The vertices are the $(kn+1)$-generators of $\Delta:$\  $V_{\Delta}=\{a_0,\dots,a_{kn}\}.$ Let $a_i,a_j\in V_{\Delta}$ be $(kn+1)$-generators.  We draw an edge $a_i\rightarrow a_j$, if $i\neq kn$ and 
$$a_j=a_{-}(a_i+n).$$
See Figure \ref{tree_example} for an example of a tree $T_{\Delta}.$
\end{definition}

\begin{figure}
\centering
\begin{tikzpicture}
\draw (0,4) node {$0$};
\draw (1,3) node {$3$};
\draw (2,2) node {$6$};
\draw (3,1) node {$8$};
\draw (3,0) node {$11$};
\draw (4,2) node {$5$};
\draw (5,3) node {$2$};
\draw [->,>=stealth] (0.2,3.8)--(0.8,3.2);
\draw [->,>=stealth] (1.2,2.8)--(1.8,2.2);
\draw [->,>=stealth] (2.2,1.8)--(2.8,1.2);
\draw [->,>=stealth] (4.8,2.8)--(4.2,2.2);
\draw [->,>=stealth] (3.8,1.8)--(3.2,1.2);
\draw [->,>=stealth] (3,0.8)--(3,0.2);
\end{tikzpicture}
\caption{The $7$-generators of the $\Gamma_{3,7}$-semimodule $\Delta:=\{0,2,3,5,6,7,\dots\}$ are equal to 
$0,2,3,5,6,8,11.$ The tree $T_{\Delta}$ is presented in the diagram.}\label{tree_example}
\end{figure}
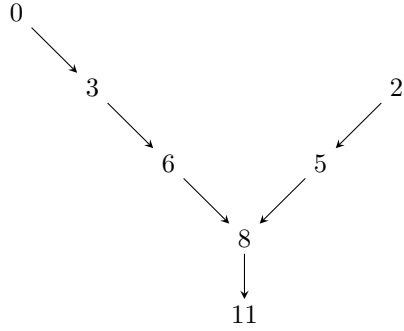

\begin{lemma}\label{tree_lemma}
The graph $T_{\Delta}$ satisfies the following properties: 
\begin{enumerate}
\item If $a_i\rightarrow a_j$ is an edge, then $i<j$. 

\item The graph $T_{\Delta}$ is a tree with the root $a_{kn}$, and all edges are oriented towards $a_{kn}$.  
 
\item The leaves of $T_{\Delta}$ are the $(kn+1,n)$-generators of $\Delta,$ i.e. they are simultaneous $n-$ and $(kn+1)-$ generators.
\end{enumerate}
\end{lemma}

\begin{proof}
\begin{enumerate}
\item By construction, $j\ge i$ and $i<kn$. If $i=j,$ then there are no $(kn+1)$-generators in the interval $[a_i+1,a_i+n],$ hence by Corollary \ref{first in a sequence is a generator} $[a_i,a_i+n]\subset \Delta$ and therefore $[a_i,+\infty)\subset \Delta$.
By Lemma \ref{arithmetic progression} there are no $(kn+1)$-generators greater than $a_i.$ Therefore $i=kn.$ Contradiction.
\item Follows from the fact that there is only one edge from each vertex except the root $a_{kn},$ and the previous part.
\item If $a_i-n\in\Delta$ then by Lemma \ref{arithmetic progression}, $a_i-n$ is a $(kn+1)$-generator. But then the tree $T_{\Delta}$ contains the edge $(a_i-n)\rightarrow a_i$ and  $a_i$ is not a leaf. Therefore, every leaf is a $(kn+1,n)$-generator. 

Conversely, suppose that $a_i$ is a $(kn+1,n)$-generator, but not a leaf of $T_{\Delta}.$ Then there exists a $(kn+1)$-generator $a_j$ such that $a_j\rightarrow a_i$ is an edge of $T_{\Delta},$ so $a_j+n>a_i$ and the interval $[a_i+1,a_j+n]$ does not contain any $(kn+1)$-generators. It then follows from  Corollary \ref{first in a sequence is a generator} that the whole interval $[a_i+1,a_j+n]$ is contained in $\Delta.$ Therefore, by Lemma \ref{first in a sequence is a generator lemma} $a_i+1-(n+1)=a_i-n\in \Delta$. Contradiction.
\end{enumerate}
\end{proof}

We will need the following observation about paths in $T_{\Delta}:$
\begin{lemma}\label{paths}
Suppose that $a_{i_0}\rightarrow a_{i_1}\rightarrow\dots\rightarrow a_{i_l}$ is a path in $T_{\Delta}.$ 
\begin{enumerate}
\item The interval $[a_{i_l}+1,a_{i_0}+ln]$ is a subset of $\Delta$ and it does not contain $(kn+1)$-generators.
\item The number of $(kn+1)$-generators in the interval $[a_{i_0}+1,a_{i_0}+ln]$ is equal to $i_l-i_0.$
\item If $a_j$ is a $(kn+1)$-generator and $a_{i_{l-1}}<a_j\le a_{i_l},$ then $N_{i_0j}=l.$
\end{enumerate}
\end{lemma}
\begin{proof}
The proof of the first part is by induction on $l.$ The $l=1$ case is by definition and Corollary \ref{first in a sequence is a generator}. Suppose that we proved the Lemma for $l-1.$ Then one has
$$
[a_{i_l}+1,a_{i_0}+ln]=[a_{i_l}+1, a_{i_{l-1}}+n]\cup ([a_{i_{l-1}}+1,a_{i_0}+(l-1)n]+n)
$$

The interval $[a_{i_l}+1, a_{i_{l-1}}+n]$ does not contain $(kn+1)$-generators by the definition of the edge $a_{l-1}\rightarrow a_l.$ In turn, the interval $([a_{i_{l-1}}+1,a_{i_0}+(l-1)n]+n)$ does not contain $(kn+1)$-generators by the induction assumption and Lemma \ref{arithmetic progression}. Finally, the inclusion $[a_{i_l}+1,a_{i_0}+ln]\subset\Delta$ follows from Corollary \ref{first in a sequence is a generator}.

The second part follows immediately from the first. For the third part, observe that since there are no $(kn+1)$-generators in the interval $[a_{i_{l-1}}+1,a_{i_0}+(l-1)n]$, we get
$$
a_{i_0}+(l-1)n<a_j\le a_{i_l}\le a_{i_0}+ln.
$$
Therefore, $N_{i_0j}=\left\lceil\frac{a_j-a_{i_0}}{n} \right\rceil=l.$
\end{proof}

\begin{theorem}\label{tree reconstruction +1}
One can reconstruct the tree $T_{\Delta}$ from the numbers $g(a_0),\dots, g(a_{kn}).$
\end{theorem}

\begin{proof}
We will reconstruct $T_{\Delta}$ in the following order. First, we reconstruct the path from $a_0$ to $a_{kn}.$ Then we take the smallest $(kn+1)$-generator $a_l,$ which is not covered yet, and reconstruct the path 
$$
a_l\rightarrow a_{l+b^{l}_0}\rightarrow a_{l+b^{l}_0+b^{l}_1}\rightarrow\dots\rightarrow a_{kn}.
$$
We repeat this procedure until we run out of generators.

On each step we need to find numbers $b^{l}_0,b^{l}_1,b^{l}_2,\dots.$ According to Lemma \ref{paths}, $b^{l}_i$ is equal to the number of $(kn+1)$-generators in the interval $I^{l}_i:=(a_l+in,a_l+(i+1)n]$. We use this to define $b^{l}_i$ for all $i\in \mathbb Z.$

We will also need numbers $c^{l}_i,$ counting $n$-cogenerators in the same intervals $I^{l}_i.$ According to Lemma \ref{numgens}, $b^{l}_i$ is equal to the number of $n$-cogenerators in the interval $J^{l}_i:=(a_l+in-(kn+1),a_l+in].$ Note that 
$$
J^{l}_i=\{a_{l}+in-kn\}\sqcup I^{l}_{i-k}\sqcup\dots \sqcup I^{l}_{i-1}
$$

By construction, $a_l$ is a leaf of the tree $T_{\Delta}.$ By Lemma \ref{tree_lemma}, $a_l$ is a $(kn+1,n)$-generator of $\Delta.$ Therefore, $a_l+in-kn=a_l+n(i-k)$ is an $n$-cogenerator if and only if  $i=k-1.$ We conclude that $b^{l}_i$ can be expressed via $c^{l}_{i-k},\dots,c^{l}_{i-1}$ as follows: 
\begin{equation}
\label{BtoC}
b^{l}_i=
\left\{
\begin{array}{l}
\sum\limits_{j=i-k}^{j=i-1} c^{l}_j, i\neq k-1 \\
1+\sum\limits_{j=i-k}^{j=i-1} c^{l}_j,i=k-1.
\end{array}
\right.
\end{equation}

On the other hand, for $i\ge 0$ one can use Lemmas \ref{numgens} and \ref{paths}  to express $c^{l}_i$ through $b^{l}_0,b^{l}_1,\dots,b^{l}_i$ using numbers $g(a_j):$ 
\begin{equation}
\label{CtoB}
c^{l}_i=g(a_{l+b^{l}_0+b^{l}_1+\dots+b^{l}_{i-1}})-g(a_{l+b^{l}_0+b^{l}_1+\dots+b^{l}_i}),\quad c^{l}_0=g(a_l)-g(a_{l+b^{l}_0}),
\end{equation}
Indeed, $g(a_{l+b^{l}_0+b^{l}_1+\dots+b^{l}_i})$ is equal to the number of integers not in $\Delta$ in the interval $[a_{l+b^{l}_0+b^{l}_1+\dots+b^{l}_i},a_{l+b^{l}_0+b^{l}_1+\dots+b^{l}_i}+n),$ which is the same as in the interval $[a_l+in, a_l+(i+1)n)$ by the first part of Lemma \ref{paths}. 

Equations (\ref{BtoC}) and (\ref{CtoB}) together provide recurrence relations on numbers $b^{l}_i$ and $c^{l}_i.$ To start the recursive algorithm, one needs to find numbers $c^{l}_{-k},\dots, c^{l}_{-1}.$

Recall the numbers $N_{ij}=\left\lceil\frac{a_j-a_i}{n}\right\rceil$ and $K_{ij}=\left\lceil\frac{a_j-a_i}{kn+1}\right\rceil.$ Since we know all edges of the tree $T_{\Delta}$ with initial points less then $a_l,$ we can use Lemma \ref{paths} to find numbers $N_{ij}$ for any $i,j\le l.$ Then, using Corollary \ref{KtoN}, we can find numbers $K_{ij}$ for any $i,j\le l.$ We can also compute
$$
f(a_i):=\sharp\left(\Delta\cap (-\infty,a_i)\right)=\sum_{j<i}K_{ji},
$$
for all $i\le l.$ Indeed, fix a remainder $0\le r<kn+1,$ such that the corresponding $(kn+1)$-generator $a_j\equiv r\ (\mod\  kn+1)$ is less than $a_i.$ Then there are exactly $K_{ji}$ integers in $\Delta\cap (-\infty,a_i)$ with remainder $r$ modulo $kn+1.$  

Note that
\begin{equation}\label{c_-i}
c^{l}_{-i}=\sharp (I^{l}_{-i+1}\cap\Delta)-\sharp(I^{l}_{-i}\cap\Delta)\  \mbox{\rm for all}\  i>0. 
\end{equation}

Since $a_l$ is a $(kn+1,n)$-generator of $\Delta,$ we have $a_l-in\notin\Delta$ for $i>0.$ Consider the smallest element of $\Delta$ bigger than $a_l-in.$ By Corollary \ref{first in a sequence is a generator}, it is a $(kn+1)$-generator $a_{\alpha_i},$ where $\alpha_{i}:=\min\{j:N_{jl}\le i\}.$

Finally, we compute
$$
\sharp (I^{l}_{-i}\cap\Delta)=f(a_{\alpha_{i-1}})-f(a_{\alpha_{i}}) \mbox{\rm for}\ i>1,
$$
$$
\sharp (I^{l}_{-1}\cap\Delta)=f(a_l)-f(a_{\alpha_1})+1,
$$
and the number $\sharp (I^{l}_0\cap\Delta)=n-g(a_l)$ is given. Therefore, we can use Equation \ref{c_-i} to find numbers $c^{l}_{-k},\dots, c^{l}_{-1}.$

\end{proof}

\begin{theorem}\label{tree -> semimod +1}
The tree $T_{\Delta}$ completely determines the semimodule $\Delta.$
\end{theorem}

\begin{proof}
Now that we know the whole tree $T_{\Delta},$ we can use Lemma \ref{paths} and Corollary \ref{KtoN} to find numbers $K_{ij}$ and $N_{ij}$ for all $i<j\le kn.$ Note that by the definition of numbers $K_{ij},$ we have
$$
a_i\in [(K_{0i}-1)(nk+1),K_{0i}(nk+1)\ )\ \mbox{\rm for all}\ i.
$$
Therefore, it suffices to recover the remainders $r_0,\dots,r_{kn}$ of the generators $a_0,\dots, a_{kn}$ modulo $kn+1.$ 

We can use numbers $K_{ij}$ recover the order of $r_0,\dots,r_{kn}.$ Indeed, if $i<j$ then
$$
r_i<r_j \ \Leftrightarrow \ K_{ij}>K_{0j}-K_{0i}.
$$
Since the remainders $r_0,\dots,r_{kn}$ run through all numbers $0,1,\dots, kn$ once, knowing the order of $r_0,\dots,r_{kn}$ is equivalent to knowing the remainders themselves.
\end{proof}

Let us illustrate Theorems \ref{tree reconstruction +1} and \ref{tree -> semimod +1} in the following example. 

\subsection{Example: reconstruction of a $\Gamma_{4,9}$-semimodule.}

Let $n=4$ and $m=9$ ($k=2$). Suppose that $g(a_0)=2,\ g(a_1)=g(a_2)=1,$ and $g(a_3)=\dots=g(a_8)=0.$ At the first step we reconstruct the path from $a_0=0$ to $a_8$ in the tree $T_{\Delta}.$ Following the algorithm, we first need to find numbers $c^{0}_{-1}$ and $c^{0}_{-2},$ counting $4$-cogenerators in intervals $(-4,0]$ and $(-8,-4]$ correspondingly. Since there are no elements of $\Delta$ less than $a_0=0,$ we immediately conclude that 
$$
c^{0}_{-2}=\sharp\{\Delta\cap (-4,0]\}-\sharp\{\Delta\cap (-8,-4]\}=1-0=1,
$$ 
and 
$$
c^{0}_{-1}=\sharp\{\Delta\cap (0,4]\}-\sharp\{\Delta\cap (-4,0]\}=(4-g(a_0))-1=4-2-1=1.
$$

Using the recurrence relations (\ref{BtoC}) and (\ref{CtoB}) we immediately compute:

$$
b^0_0=c^{0}_{-2}+c^{0}_{-1}=2,\ c^{0}_0=g(a_0)-g(a_2)=2-1=1,
$$
$$
b^0_1=1+c^{0}_{-1}+c^{0}_0=3,\ c^{0}_1=g(a_2)-g(a_5)=1,
$$
$$
b^0_2=c^{0}_0+c^{0}_1=2,\ c^{0}_2=g(a_5)-g(a_7)=0,
$$
$$
b^0_3=c^{0}_1+c^{0}_2=1,\ c^{0}_3=g(a_7)-g(a_8)=0.
$$

Therefore, the path from $a_0$ to $a_8$ is $a_0\rightarrow a_2\rightarrow a_5\rightarrow a_7\rightarrow a_8.$ The smallest $9$-generator not covered yet is $a_1.$ Therefore, our next step is to recover the path from $a_1$ to $a_8.$

Again, we start by reconstructing numbers $c^{1}_{-1}$ and $c^{1}_{-2}.$ We have
$$
c^{1}_{-2}=\sharp\{\Delta\cap (a_1-4,a_1]\}-\sharp\{\Delta\cap (a_1-8,a_1-4]\},
$$ 
and 
$$
c^{1}_{-1}=\sharp\{\Delta\cap (a_1,a_1+4]\}-\sharp\{\Delta\cap (a_1-4,a_1]\}.
$$
The only element of $\Delta$ less then $a_1$ is $a_0$ and, moreover, $a_0>a_1-4.$ Indeed, $a_0$ is the only $9$-generator less than $a_1$ and $a_1<a_2\le a_0+4,$ because we have the arrow $a_0\rightarrow a_2$ in the tree $T_{\Delta}.$ We conclude that
$$
\sharp\{\Delta\cap (a_1-8,a_1-4]\}=0,
$$
$$
\sharp\{\Delta\cap (a_1-4,a_1]\}=2,
$$
and
$$
\sharp\{\Delta\cap (a_1,a_1+4]\}=4-g(a_1)=3.
$$

Therefore,
$$
c^{1}_{-2}=2,\ \mbox{and}\ c^{1}_{-1}=1.
$$

We again use the recurrence relations (\ref{BtoC}) and (\ref{CtoB}):

$$
b^{1}_0=c^{1}_{-2}+c^{1}_{-1}=3,\ c^{1}_0=g(a_1)-g(a_4)=1,
$$
$$
b^{1}_1=1+c^{1}_{-1}+c^{1}_0=3,\ c^{1}_1=g(a_4)-g(a_7)=0,
$$
$$
b^{1}_2=c^{1}_0+c^{1}_1=1,\ c^{1}_2=g(a_7)-g(a_8)=0.
$$

Therefore, the path from $a_1$ to $a_8$ is $a_1\rightarrow a_4\rightarrow a_7\rightarrow a_8.$ The smallest $9$-generator not covered yet is $a_3.$ Therefore, our next step is to recover the path from $a_3$ to $a_8.$

Similarly to above, we need to reconstruct numbers $c^{3}_{-2}$ and $c^{3}_{-1}$ first: 
$$
c^{3}_{-2}=\sharp\{\Delta\cap (a_3-4,a_3]\}-\sharp\{\Delta\cap (a_3-8,a_3-4]\},
$$ 
and 
$$
c^{3}_{-1}=\sharp\{\Delta\cap (a_3,a_3+4]\}-\sharp\{\Delta\cap (a_3-4,a_3]\}.
$$
There are three $9$-generators less than $a_3:$ $a_0,a_1,$ and $a_2.$ Since we have the path $a_0\rightarrow a_2\rightarrow a_5$ in the tree $T_{\Delta},$ we conclude that $a_0+8\ge a_5>a_3.$ Therefore, the only $3$ elements of $\Delta$ less than $a_3$ are the $9$-generators. Furthermore, from the reconstructed part of the tree we see that $a_3-8<a_0<a_3-4,$ $a_3-4<a_1<a_3,$ and $a_3-4<a_2<a_3.$ Therefore,

$$
\sharp\{\Delta\cap (a_3-8,a_3-4]\}=1,
$$
$$
\sharp\{\Delta\cap (a_3-4,a_3]\}=3,
$$
and
$$
\sharp\{\Delta\cap (a_3,a_3+4]\}=4-g(a_3)=4.
$$

We conclude that
$$
c^{3}_{-2}=2,\ \mbox{and}\ c^{3}_{-1}=1.
$$

Once again, we use the recurrence relations (\ref{BtoC}) and (\ref{CtoB}):

$$
b^{3}_0=c^{3}_{-2}+c^{3}_{-1}=3,\ c^{3}_0=g(a_3)-g(a_6)=0,
$$
$$
b^{3}_1=1+c^{3}_{-1}+c^{3}_0=2,\ c^{3}_1=g(a_6)-g(a_8)=0.
$$

Therefore, the path from $a_3$ to $a_8$ is $a_3\rightarrow a_6\rightarrow a_8.$ See Figure \ref{tree_rec_example} for the full tree $T_{\Delta}.$

\begin{figure}
\centering
\begin{tikzpicture}
\draw (0,4) node {$a_0$};
\draw (1,3) node {$a_2$};
\draw (2,2) node {$a_5$};
\draw (3,1) node {$a_7$};
\draw (4,0) node {$a_8$};
\draw (4,2) node {$a_4$};
\draw (5,3) node {$a_1$};
\draw (5,1) node {$a_6$};
\draw (6,2) node {$a_3$};
\draw [->,>=stealth] (0.2,3.8)--(0.8,3.2);
\draw [->,>=stealth] (1.2,2.8)--(1.8,2.2);
\draw [->,>=stealth] (2.2,1.8)--(2.8,1.2);
\draw [->,>=stealth] (3.2,0.8)--(3.8,0.2);
\draw [->,>=stealth] (4.8,2.8)--(4.2,2.2);
\draw [->,>=stealth] (3.8,1.8)--(3.2,1.2);
\draw [->,>=stealth] (5.8,1.8)--(5.2,1.2);
\draw [->,>=stealth] (4.8,0.8)--(4.2,0.2);
\end{tikzpicture}
\caption{The tree $T_{\Delta}$ for the case $n=4,\ m=9,\ g(a_0)=2,\ g(a_1)=g(a_2)=1,$ and $g(a_3)=\dots=g(a_8)=0.$}\label{tree_rec_example}
\end{figure}
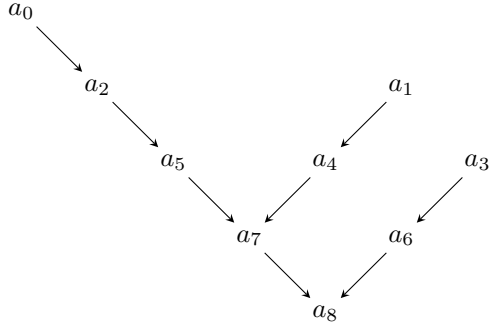

Finally, we reconstruct the semimodule $\Delta$ from the tree $T_{\Delta}$ following the algorithm from the Theorem \ref{tree -> semimod +1}. First, we use Lemma \ref{paths} to find numbers $N_{0,i}$ for $i=1,2,\dots, 8$ and Corollary \ref{KtoN} to find $K_{0,i}.$ We get

$$
N_{0,1}=N_{0,2}=1,\ N_{0,3}=N_{0,4}=N_{0,5}=2,\ N_{0,6}=N_{0,7}=3,\ N_{0,8}=4,
$$
and
$$
K_{0,1}=K_{0,2}=K_{0,3}=K_{0,4}=K_{0,5}=1,\ K_{0,6}=K_{0,7}=K_{0,8}=2.
$$

Therefore, we get that $a_1,a_2,a_3,a_4,$ and $a_5$ are in the interval $(0,9),$ while $a_6,a_7,a_8$ are in the interval $(9,18).$

Now we need to compare the remainders $r_1,\dots, r_8$ of the $9$-generators $a_1,\dots,a_8.$ We already know that $r_1<r_2<r_3<r_4<r_5$ and $r_6<r_7<r_8.$ Let us compare $r_1$ and $r_6.$ By Lemma \ref{paths} and the tree $T_{\Delta},$ we get $N_{1,6}=2.$ Therefore, by Corollary \ref{KtoN}, $K_{1,6}=1,$ which means that $a_6-a_1<9.$ Therefore, $r_6=a_6-9<a_1=r_1.$

Similarly, one computes that $r_7<r_1$ and $r_2<r_8<r_3.$ Therefore, 
$$
r_6<r_7<r_1<r_2<r_8<r_3<r_4<r_5,
$$ 
or 
$$
r_6=1,\ r_7=2,\ r_1=3,\ r_2=4,\ r_8=5,\ r_3=6,\ r_4=7,\ r_5=8.
$$
Finally,
$$
a_0=0,\ a_1=r_1=3,\ a_2=r_2=4,\ a_3=r_3=6,\ a_4=r_4=7,\ a_5=r_5=8,
$$
and
$$
a_6=r_6+9=10,\ a_7=r_7+9=11,\ a_8=r_8+9=14.
$$

\subsection{The case $m=kn-1.$}
\label{SS kn-1}

Let now $m=kn-1.$ In this case the semimodule $\Delta$ can be reconstructed from numbers $g(a_0),\dots, g(a_{kn-2})$ in a way similar to the case $m=kn+1.$ However, some adjustments are required. We will omit some of the proofs, in the cases when they are identical to the $m=kn+1$ case.

\begin{lemma}\label{last in a sequence is a generator lemma}
Suppose that $x\in\Delta$ and $x-\alpha n+1\notin\Delta$ for some $\alpha\in\{0,\dots,k\}$. Then $x$ is a $(kn-1)$-generator.
\end{lemma}

\begin{proof}
Indeed, if $x$ is not a $(kn-1)$-generator, than $x-kn+1\in\Delta.$ But then $(x-kn+1)+(k-\alpha)n=x-\alpha n+1\in\Delta.$ Contradiction. 
\end{proof}

\begin{definition}
For $x\in \Delta$ we define $a_{+}(x)$ to be the minimal $(kn-1)$-generator greater than or equal to $x$. If there is no $(kn-1)$-generators greater or equal to $x,$ we set $a_+(x)=\infty.$
\end{definition}

\begin{corollary}\label{last in a sequence is a generator}
For any $x\in \Delta$ one has $[x,a_{+}(x)]\subset \Delta$.
\end{corollary}

\begin{proof}
Follows from Lemma \ref{last in a sequence is a generator lemma} with $\alpha=0.$
\end{proof}

\begin{lemma}\label{gaps-1}
Consider a $(kn-1)$-generator $a_i.$ Then for any $l>0,$ the interval $J_l:=[a_i-lkn+1,a_i-l(kn-1)]$ has empty intersection with $\Delta.$
\end{lemma}

\begin{proof}
The proof is by induction on $l.$ The case $l=1$ is clear. Then observe that 
$$
J_{l+1}=\{a_i-(l+1)(kn-1)\}\cup (J_l-kn). 
$$ 
\end{proof}

We illustrate Lemma \ref{gaps-1} with the Figure \ref{gapspic-1}.

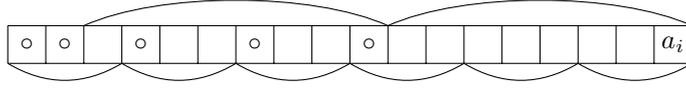
\begin{figure}
\centering

\begin{tikzpicture}[scale=0.5]
\draw (0,0) grid +(18,1);

\draw (0.5,0.5) node {$\circ$};
\draw (1.5,0.5) node {$\circ$};
\draw (3.5,0.5) node {$\circ$};
\draw (6.5,0.5) node {$\circ$};
\draw (9.5,0.5) node {$\circ$};
\draw (17.5,0.5) node {$a_i$};

\draw (0,0) .. controls (1,-0.6) and (2,-0.6) .. (3,0);
\draw (3,0) .. controls (4,-0.6) and (5,-0.6) .. (6,0);
\draw (6,0) .. controls (7,-0.6) and (8,-0.6) .. (9,0);
\draw (9,0) .. controls (10,-0.6) and (11,-0.6) .. (12,0);
\draw (12,0) .. controls (13,-0.6) and (14,-0.6) .. (15,0);
\draw (15,0) .. controls (16,-0.6) and (17,-0.6) .. (18,0);

\draw (2,1) .. controls (4,1.9) and (8,1.9) .. (10,1);
\draw (10,1) .. controls (12,1.9) and (16,1.9) .. (18,1);

\end{tikzpicture}
\caption{In this picture $n=3,$ $k=3,$ $a_i$ is an $8$-generator, and "$\circ$"\ indicates an integer not in $\Delta.$ Note that the interval $[a_i-2kn+1,a_i-2(kn-1)]=[a_i-17,a_i-16]$ has empty intersection with $\Delta.$}
\label{gapspic-1}
\end{figure}

\begin{definition}
We introduce the following notations:
$$
N_{ij}=\left\lfloor\frac{a_j-a_i}{n} \right\rfloor+1\mbox{ and }
K_{ij}=\left\lfloor\frac{a_j-a_i}{kn-1} \right\rfloor+1.
$$
\end{definition}

\begin{corollary}\label{KtoN-1}
We have the following formula:
$
K_{ij}=\left\lceil\frac{N_{ij}}{k} \right\rceil.
$
\end{corollary}

\begin{proof}
Note that 
$$
\left\lceil\frac{N_{ij}}{k} \right\rceil=\left\lceil\frac{\left\lfloor\frac{a_j-a_i}{n} \right\rfloor+1}{k} \right\rceil=\left\lfloor\frac{a_j-a_i}{kn} \right\rfloor+1.
$$ 
Suppose that  $K_{ij}\neq\left\lceil\frac{N_{ij}}{k} \right\rceil.$ Then there exists $l,$ such that $a_j-a_i\ge l(kn-1)$ and $a_j-a_i<lkn,$ which is equivalent to $a_i\in (a_j-lkn,a_j-l(kn-1)].$ Therefore, by Lemma \ref{gaps-1} $a_i\notin \Delta.$ Contradiction.
\end{proof}


\begin{definition}
\label{def tree-1}
The oriented graph $T_{\Delta}$ is defined as follows. The vertices are the $(kn-1)$-generators of $\Delta$ plus one extra vertex $a_{\infty}:$\  $V_{\Delta}=\{a_0,\dots,a_{kn-2},a_{\infty}\}.$ Let $a_i, a_j\in V_{\Delta}$ be $(kn-1)$-generators.  We draw an edge $a_i\rightarrow a_j$, if
$$
a_j=a_{+}(a_i+n).
$$
In addition, we draw edges $a_i\rightarrow a_{\infty}$ for every $a_i,$ such that there is no $(kn-1)$-generators greater than or equal to $a_i+n.$ See Figure \ref{tree-1} for an example of a tree $T_{\Delta}.$
\end{definition}

\begin{figure}
\centering
\begin{tikzpicture}
\draw (0,4) node {$0$};
\draw (1,3) node {$3$};
\draw (2,2) node {$6$};
\draw (3,1) node {$9$};
\draw (4,0) node {$13$};
\draw (4,-1) node {$\infty$};
\draw (5,1) node {$10$};
\draw (6,2) node {$7$};
\draw (7,3) node {$4$};

\draw [->,>=stealth] (0.2,3.8)--(0.8,3.2);
\draw [->,>=stealth] (1.2,2.8)--(1.8,2.2);
\draw [->,>=stealth] (2.2,1.8)--(2.8,1.2);
\draw [->,>=stealth] (3.2,0.8)--(3.8,0.2);

\draw [->,>=stealth] (6.8,2.8)--(6.2,2.2);
\draw [->,>=stealth] (5.8,1.8)--(5.2,1.2);
\draw [->,>=stealth] (4.8,0.8)--(4.2,0.2);

\draw [->,>=stealth] (4,-0.2)--(4,-0.8);
\end{tikzpicture}
\caption{The $8$-generators of the $\Gamma_{3,8}$-semimodule $\Delta:=\{0,3,4,6,7,8,\dots\}$ are equal to $0,3,4,6,7,9,10,13.$ The tree $T_{\Delta}$ is presented in the diagram.}
\label{tree-1}
\end{figure}
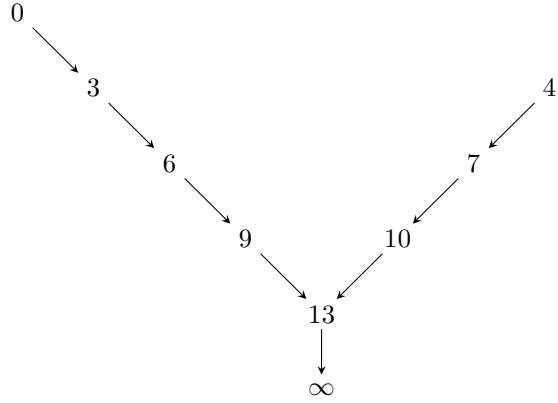

\begin{lemma}\label{tree_lemma-1} The graph $T_{\Delta}$ satisfies the following properties: 

\begin{enumerate}
\item If $a_i\rightarrow a_j$ is an edge, then $i<j$. 

\item The graph $T_{\Delta}$ is a tree with the root $a_{\infty}$, and all edges are oriented towards $a_{\infty}$.  
 
\item The leaves of $T_{\Delta}$ are the $(kn-1,n)$-generators of $\Delta$.
\end{enumerate}
\end{lemma}

\begin{proof}
The first two parts are immediate from the construction.  For the third part, consider a $(kn-1)$-generator $a_i.$ If $a_i-n\in\Delta$ then by Lemma \ref{arithmetic progression} $a_i-n$ is a $(kn-1)$-generator. But then the tree $T_{\Delta}$ contains the edge $(a_i-n)\rightarrow a_i$ and  $a_i$ is not a leaf. Therefore, every leaf is a $(kn-1,n)$-generator. 

Conversely, suppose that $a_i$ is a $(kn-1,n)$-generator, but not a leaf of $T_{\Delta}.$ Then there exists a $(kn-1)$-generator $a_j$ such that $a_j\rightarrow a_i$ is an edge of $T_{\Delta},$ so $a_j+n<a_i$ and the interval $[a_j+n,a_i-1]$ does not contain any $(kn-1)$-generators. It then follows from  Corollary \ref{last in a sequence is a generator} that the whole interval $[a_j+n,a_i-1]$ is contained in $\Delta.$ Therefore by Lemma \ref{last in a sequence is a generator lemma} $a_i-1-(n-1)=a_i-n\in \Delta$. Contradiction.
\end{proof}

We will also need the following observation about paths in $T_{\Delta}:$

\begin{lemma}\label{paths-1}
Suppose $a_{i_0}\rightarrow a_{i_1}\rightarrow\dots\rightarrow a_{i_l}$ is a path in $T_{\Delta}.$ 
\begin{enumerate}
\item The interval $[a_{i_0}+ln,a_{i_l}-1]$ is a subset of $\Delta$ and it does not contain $(kn-1)$-generators.
\item The number of $(kn-1)$-generators in the interval $[a_{i_0},a_{i_0}+ln-1]$ is equal to $i_l-i_0.$
\item If $a_j$ is a $(kn-1)$-generator and $a_{i_{l-1}}\le a_j< a_{i_l},$ then $N_{i_0j}=l.$
\end{enumerate}
\end{lemma}

\begin{proof}
The proof of the first part is by induction on $l.$ The $l=1$ case is by definition and Corollary \ref{last in a sequence is a generator}. Suppose that we proved the Lemma for $l-1.$ Then one has
$$
[a_{i_0}+ln,a_{i_l}-1]= ([a_{i_0}+(l-1)n,a_{i_{l-1}}-1]+n)\cup [a_{i_{l-1}}+n,a_{i_l}-1]
$$

The interval $[a_{i_{l-1}}+n,a_{i_l}-1]$ does not contain $(kn-1)$-generators by the definition of the edge $a_{l-1}\rightarrow a_l.$ In turn, the interval $ ([a_{i_0}+(l-1)n,a_{i_{l-1}}-1]+n)$ does not contain $(kn-1)$-generators by the induction assumption and Lemma \ref{arithmetic progression}. Finally, the inclusion $[a_{i_0}+ln,a_{i_l}-1]\subset\Delta$ follows from Corollary \ref{last in a sequence is a generator}.

The second part follows immediately from the first. For the third part, observe that since there are no $(kn-1)$-generators in the interval $[a_{i_0}+ln, a_{i_l}-1]$, we get
$$
a_{i_0}+(l-1)n\le a_{i_{l-1}}\le a_j< a_{i_0}+ln.
$$
Therefore, $N_{i_0j}=\left\lfloor\frac{a_j-a_{i_0}}{n} \right\rfloor+1=l-1+1=l.$
\end{proof}

\begin{theorem}
One can reconstruct $T_{\Delta}$ from numbers $g(a_0),\dots, g(a_{kn-2}).$
\end{theorem}

\begin{proof}
We will reconstruct $T_{\Delta}$ in the following order. First, we reconstruct the path from $a_0$ to $a_{\infty}.$ Then we take the smallest $(kn-1)$-generator $a_l,$ which is not covered yet, and reconstruct the path 
$$
a_l\rightarrow a_{l+b^{l}_0}\rightarrow a_{l+b^{l}_0+b^{l}_1}\rightarrow\dots\rightarrow a_{\infty}.
$$
We repeat this procedure until we run out of generators.

On each step all we need to do is to find numbers $b^{l}_0,b^{l}_1,b^{l}_2,\dots.$ According to Lemma \ref{paths-1}, $b^{l}_i$ is equal to the number of $(kn-1)$-generators in the interval $I^{l}_i:=[a_l+in,a_l+(i+1)n)$. We use this to define $b^{l}_i$ for all $i\in\mathbb Z.$

We will also need numbers $c^{l}_i,$ counting $n$-cogenerators in the same intervals $I^{l}_i.$
According to Lemma \ref{numgens}, $b^{l}_i$ is equal to the number of $n$-cogenerators in the interval $J^{l}_i:=[a_l+in-(kn-1),a_l+in).$ Note that 
$$
J^{l}_i=(I^{l}_{i-k}\sqcup\dots \sqcup I^{l}_{i-1})\backslash (a_l+(i-k)n)
$$

By construction, $a_l$ is a leaf of the tree $T_{\Delta}.$ By Lemma \ref{tree_lemma-1}, $a_l$ is a $(kn-1,n)$-generator of $\Delta.$ Therefore $a_l+in-kn=a_l+n(i-k)$ is an $n$-cogenerator if and only if $i=k-1.$ We conclude that $b^{l}_i$ can be expressed via $c^{l}_{i-k},\dots,c^{l}_{i-1}$ as follows:

\begin{equation}
\label{BtoC-1}
b^{l}_i=
\left\{
\begin{array}{l}
\sum\limits_{j=i-k}^{j=i-1} c^{l}_j, i\neq k-1 \\
-1+\sum\limits_{j=i-k}^{j=i-1} c^{l}_j,i=k-1.
\end{array}
\right.
\end{equation}

On the other hand, for $i\ge 0$ one can use Lemmas \ref{numgens} and \ref{paths-1} to express $c^{l}_i$ through $b^{l}_0,b^{l}_1,b^{l}_i$ using numbers $g(a_j):$
\begin{equation}
\label{CtoB-1}
c^{l}_i=g(a_{l+b^{l}_0+b^{l}_1+\dots+b^{l}_{i-1}})-g(a_{l+b^{l}_0+b^{l}_1+\dots+b^{l}_i}),\quad c^{l}_0=g(a_l)-g(a_{l+b^{l}_0}),
\end{equation}
Indeed, $g(a_{l+b^{l}_0+b^{l}_1+\dots+b^{l}_i})$ is equal to the number of integers not in $\Delta$ in the interval $[a_{l+b^{l}_0+b^{l}_1+\dots+b^{l}_i},a_{l+b^{l}_0+b^{l}_1+\dots+b^{l}_i}+n),$ which is the same as in the interval $[a_l+in, a_l+(i+1)n)$ by the first part of Lemma \ref{paths-1}.

Equations (\ref{BtoC-1}) and (\ref{CtoB-1}) together provide recurrence relations on numbers $b^{l}_i$ and $c^{l}_i.$ To start the recursive algorithm, one needs to find numbers $c^{l}_{-k},\dots, c^{l}_{-1}.$

Recall the numbers $N_{ij}=\lfloor\frac{a_j-a_i}{n}\rfloor+1$ and $K_{ij}=\lfloor\frac{a_j-a_i}{kn+1}\rfloor+1.$ Since we know all edges of the tree $T_{\Delta}$ with initial points less then $a_l,$ we can use Lemma \ref{paths-1} to find numbers $N_{ij}$ for any $i,j\le l.$ Then, by Corollary \ref{KtoN-1} we can find numbers $K_{ij}$ for any $i,j\le l.$ We can also compute
$$
f(a_i):=\sharp\left(\Delta\cap (-\infty,a_i)\right]=\sum_{j\le i}K_{ji}.
$$
Indeed, fix a remainder $0\le r<kn-1,$ such that the corresponding $(kn-1)$-generator $a_j\equiv r\ (\mod\  kn-1)$ is less than or equal to $a_i.$ Then there are exactly $K_{ji}$ integers in $\Delta\cap (-\infty,a_i]$ with remainder $r$ modulo $kn-1.$  

Note that
\begin{equation}\label{c_-i-1}
c^{l}_{-i}=\sharp (I^{l}_{-i+1}\cap\Delta)-\sharp(I^{l}_{-i}\cap\Delta)\  \mbox{\rm for all}\  i>0.
\end{equation}

Since $a_l$ is a $(kn-1,n)$-generator of $\Delta,$ we have $a_l-in\notin\Delta$ for $i>0.$ Consider the biggest element of $\Delta$ less than $a_l-in.$ By Corollary \ref{last in a sequence is a generator}, it is a $(kn-1)$-generator $a_{\alpha_i},$ where $\alpha_{i}:=\max\{j|N_{jl}\> i\}.$ If the set $\{j|N_{jl}> i\}$ is empty, then $a_l-in\le 0$ and there is no elements of $\Delta$ less then $a_l-in.$ For consistency, we set $\alpha_i=-1,$ $a_{-1}=-\infty,$ and $f(-\infty)=0$ in this case. 

Finally, we compute
$$
\sharp (I^{l}_{-i}\cap\Delta)=f(a_{\alpha_{i-1}})-f(a_{\alpha_{i}}) \mbox{\rm for}\ i>1,
$$
$$
\sharp (I^{l}_{-1}\cap\Delta)=f(a_l)-f(a_{\alpha_1})-1,
$$
and the number $\sharp (I^{l}_0\cap\Delta)=n-g(a_l)$ is given. Therefore, we can use Equation \ref{c_-i-1} to find numbers $c^{l}_{-k},\dots, c^{l}_{-1}.$
\end{proof}

\begin{theorem}
The tree $T_{\Delta}$ completely determines the semimodule $\Delta.$
\end{theorem}

\begin{proof}
The same as in the $m=kn+1$ case.
\end{proof}

\subsection{Example: reconstruction of a $\Gamma_{5,9}$-semimodule.}

Let $n=5,$ $k=2,$ and $m=2\times 5-1=9.$ Suppose that $g(a_0)=3,\ g(a_1)=g(a_2)=g(a_3)=2,$ $g(a_4)=1,$ and $g(a_5)=\dots=g(a_8)=0.$ At the first step we reconstruct the path from $a_0=0$ to $a_{\infty}$ in the tree $T_{\Delta}.$ Following the algorithm, we first need to find numbers $c^{0}_{-1}$ and $c^{0}_{-2},$ counting $5$-cogenerators in intervals $[-5,0)$ and $[-10,-5)$ correspondingly. Since there are no elements of $\Delta$ less than $a_0=0,$ we immediately conclude that 
$$
c^{0}_{-2}=\sharp\{\Delta\cap [-5,0)\}-\sharp\{\Delta\cap [-10,-5)\}=0-0=0,
$$ 
and 
$$
c^{0}_{-1}=\sharp\{\Delta\cap [0,5)\}-\sharp\{\Delta\cap [-5,0)\}=(5-g(a_0))-0=2.
$$

Using the recurrence relations (\ref{BtoC-1}) and (\ref{CtoB-1}) we immediately compute:

$$
b^0_0=c^{0}_{-2}+c^{0}_{-1}=2,\ c^{0}_0=g(a_0)-g(a_2)=3-2=1,
$$
$$
b^0_1=-1+c^{0}_{-1}+c^{0}_0=2,\ c^{0}_1=g(a_2)-g(a_4)=2-1=1,
$$
$$
b^0_2=c^{0}_0+c^{0}_1=2,\ c^{0}_2=g(a_4)-g(a_6)=1,
$$
$$
b^0_3=c^{0}_1+c^{0}_2=2,
$$

Therefore, the path from $a_0$ to $a_{\infty}$ is $a_0\rightarrow a_2\rightarrow a_4\rightarrow a_6\rightarrow a_8\rightarrow a_{\infty}.$ The smallest $9$-generator not covered yet is $a_1.$ Therefore, our next step is to recover the path from $a_1$ to $a_\infty.$

Again, we start by reconstructing numbers $c^{1}_{-1}$ and $c^{1}_{-2}.$ We have
$$
c^{1}_{-2}=\sharp\{\Delta\cap [a_1-5,a_1)\}-\sharp\{\Delta\cap [a_1-10,a_1-5)\},
$$ 
and 
$$
c^{1}_{-1}=\sharp\{\Delta\cap [a_1,a_1+5)\}-\sharp\{\Delta\cap [a_1-5,a_1)\}.
$$

The only $9$-generator of $\Delta$ less then $a_1$ is $a_0$ and, moreover, $a_0>a_1-5.$ In fact, $a_0+5=a_2>a_1.$ Indeed, $a_0+5=5$ is a $9$-generator, and we have the arrow $a_0\rightarrow a_2$ in the tree $T_{\Delta}.$ We conclude that
$$
\sharp\{\Delta\cap [a_1-10,a_1-5)\}=0,
$$
$$
\sharp\{\Delta\cap [a_1-5,a_1)\}=1,
$$
and
$$
\sharp\{\Delta\cap [a_1,a_1+5)\}=5-g(a_1)=3.
$$

Therefore,
$$
c^{1}_{-2}=1,\ \mbox{and}\ c^{1}_{-1}=2.
$$

We again use the recurrence relations (\ref{BtoC-1}) and (\ref{CtoB-1}):

$$
b^{1}_0=c^{1}_{-2}+c^{1}_{-1}=3,\ c^{1}_0=g(a_1)-g(a_4)=2-1=1,
$$
$$
b^{1}_1=-1+c^{1}_{-1}+c^{1}_0=2,\ c^{1}_1=g(a_4)-g(a_6)=1,
$$
$$
b^{1}_2=c^{1}_0+c^{1}_1=2,\ c^{1}_2=g(a_6)-g(a_8)=0.
$$

Therefore, the path from $a_1$ to $a_{\infty}$ is $a_1\rightarrow a_4\rightarrow a_6\rightarrow a_8\rightarrow a_{\infty}.$ The smallest $9$-generator not covered yet is $a_3.$ Therefore, our next step is to recover the path from $a_3$ to $a_{\infty}.$

Similarly to the above, we need to reconstruct numbers $c^{3}_{-2}$ and $c^{3}_{-1}$ first: 
$$
c^{3}_{-2}=\sharp\{\Delta\cap [a_3-5,a_3)\}-\sharp\{\Delta\cap [a_3-10,a_3-5)\},
$$ 
and 
$$
c^{3}_{-1}=\sharp\{\Delta\cap [a_3,a_3+5)\}-\sharp\{\Delta\cap [a_3-5,a_3)\}.
$$
There are three $9$-generators less than $a_3:$ $a_0,\ a_1,$ and $a_2.$ Since we already know the path $a_0\rightarrow a_2\rightarrow a_4,$ we immediately conclude that $a_3<a_0+10.$ Since $a_3$ and $a_0$ has different remainders modulo $9,$ we get $a_3<a_0+9.$ Therefore, there is no elements of $\Delta$ less than or equal to $a_3-9.$ We conclude that $a_0,\ a_1,$ and $a_2$ are the only elements of $\Delta$ less than $a_3.$ Moreover, $a_2+5>a_3$ and $a_1+5>a_3,$ while $a_0+5=a_2<a_3<a_0+10.$ Therefore,

$$
\sharp\{\Delta\cap [a_3-10,a_3-5)\}=1,
$$
$$
\sharp\{\Delta\cap [a_3-5,a_3)\}=2,
$$
and
$$
\sharp\{\Delta\cap [a_3,a_3+5)\}=5-g(a_3)=3.
$$

We conclude that
$$
c^{3}_{-2}=1,\ \mbox{and}\ c^{3}_{-1}=1.
$$

Once again, we use the recurrence relations (\ref{BtoC-1}) and (\ref{CtoB-1}):

$$
b^{3}_0=c^{3}_{-2}+c^{3}_{-1}=2,\ c^{3}_0=g(a_3)-g(a_5)=2,
$$
$$
b^{3}_1=-1+c^{3}_{-1}+c^{3}_0=2,\ c^{3}_1=g(a_5)-g(a_7)=0.
$$
$$
b^{3}_2=c^{3}_{0}+c^{3}_{1}=2.
$$
Therefore, the path from $a_3$ to $a_{\infty}$ is $a_3\rightarrow a_5\rightarrow a_7\rightarrow a_{\infty}.$ See Figure \ref{tree_rec_example-1} for the full tree $T_{\Delta}.$

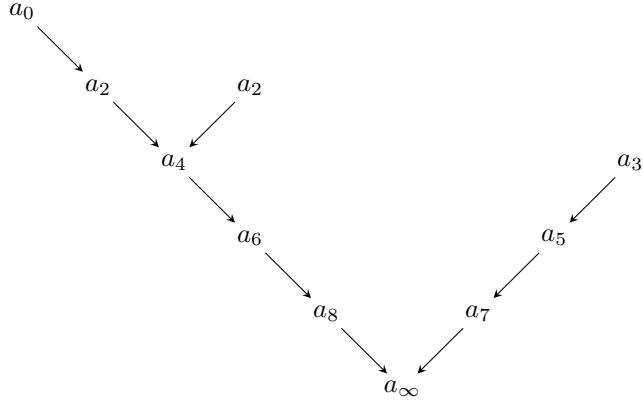
\begin{figure}
\centering
\begin{tikzpicture}
\draw (0,4) node {$a_0$};
\draw (1,3) node {$a_2$};
\draw (2,2) node {$a_4$};
\draw (3,1) node {$a_6$};
\draw (4,0) node {$a_8$};
\draw (5,-1) node {$a_{\infty}$};

\draw (8,2) node {$a_3$};
\draw (7,1) node {$a_5$};
\draw (6,0) node {$a_7$};

\draw (3,3) node {$a_2$};

\draw [->,>=stealth] (0.2,3.8)--(0.8,3.2);
\draw [->,>=stealth] (1.2,2.8)--(1.8,2.2);
\draw [->,>=stealth] (2.2,1.8)--(2.8,1.2);
\draw [->,>=stealth] (3.2,0.8)--(3.8,0.2);
\draw [->,>=stealth] (4.2,-0.2)--(4.8,-0.8);

\draw [->,>=stealth] (2.8,2.8)--(2.2,2.2);

\draw [->,>=stealth] (7.8,1.8)--(7.2,1.2);
\draw [->,>=stealth] (6.8,0.8)--(6.2,0.2);
\draw [->,>=stealth] (5.8,-0.2)--(5.2,-0.8);
\end{tikzpicture}
\caption{The tree $T_{\Delta}$ for the case $n=5,\ m=9,$ $g(a_0)=3,\ g(a_1)=g(a_2)=g(a_3)=2,$ $g(a_4)=1,$ and $g(a_5)=\dots=g(a_8)=0.$}\label{tree_rec_example-1}
\end{figure}

Finally, we reconstruct the semimodule $\Delta$ from the tree $T_{\Delta}$ following the algorithm from the Theorem \ref{tree -> semimod +1}. First, we use Lemma \ref{paths-1} to find numbers $N_{0,i}$ for $i=1,2,\dots, 8$ and Corollary \ref{KtoN-1} to find $K_{0,i}.$ We get

$$
N_{0,1}=1,\ N_{0,2}=N_{0,3}=2,\ N_{0,4}=N_{0,5}=3,\ N_{0,6}=N_{0,7}=4,\ N_{0,8}=5,
$$
and
$$
K_{0,1}=K_{0,2}=K_{0,3}=1,\ K_{0,4}=K_{0,5}=K_{0,6}=K_{0,7}=2,\ K_{0,8}=3.
$$

Therefore, we get that $a_1,a_2,$ and $a_3$ are in the interval $(0,9),$\ $a_4,a_5,a_6,a_7$ are in the interval $(9,18),$ and $a_8$ is in the interval $(18,27).$

Now we need to compare the remainders $r_1,\dots, r_8$ of the $9$-generators $a_1,\dots,a_8.$ We already know that $r_1<r_2<r_3,$ and $r_4<r_5<r_6<r_7.$ Let us compare $r_1$ and $r_4.$ By Lemma \ref{paths-1} and the tree $T_{\Delta},$ we get $N_{1,4}=2.$ Therefore, by Corollary \ref{KtoN}, $K_{1,4}=1,$ which means that $a_4-a_1<9.$ Therefore, $r_4=a_4-9<a_1=r_1.$

Similarly, one computes that $r_5<r_1,$\ $r_2<r_6<r_3,$ and $r_7>r_3.$ Therefore, we get 
$$
r_4<r_5<r_1<r_2<r_6<r_3<r_7.
$$

To compare $r_8$ with the rest of the remainders, we compute
$$
N_{1,8}=N_{2,8}=4,\ N_{3,8}=N_{4,8}=3,\ N_{5,8}=N_{6,8}=2,\ N_{7,8}=1,
$$
and
$$
K_{1,8}=K_{2,8}=K_{3,8}=K_{4,8}=2,\ K_{5,8}=K_{6,8}=K_{7,8}=1.
$$

So, $r_8=a_8-18<a_i-9=r_i$ for $i=5,6,7,$\ $r_8=a_8-18<a_i=r_i$ for $i=1,2,3,$ and $r_8=a_8-18>a_4-9=r_4.$ Therefore, 

$$
r_4<r_8<r_5<r_1<r_2<r_6<r_3<r_7,
$$
or 
$$
r_4=1,\ r_8=2,\ r_5=3,\ r_1=4,\ r_2=5,\ r_6=6,\ r_3=7,\ r_7=8.
$$
Finally,
$$
a_0=0,\ a_1=r_1=4,\ a_2=r_2=5,\ a_3=r_3=7,
$$
$$
a_4=r_4+9=10,\ a_5=r_5+9=12,\ a_6=r_6+9=15,\ a_7=r_7+9=17,
$$
and
$$
a_8=r_8+18=20.
$$

\subsection{Bounce path and statistic.}
\label{SS bounce}

In the case $m=kn+1$ the path $a_0\rightarrow a_{b^{0}_0}\rightarrow a_{b^{0}_0+b^{0}_1}\rightarrow\dots\rightarrow a_{b^{0}_0+\dots +b^{0}_s}$ in the tree $T_{\Delta}$ can be compared with the {\it bounce path}, constructed by J. Haglund in the case $m=n+1$ and generalized by N. Loehr for the case $m=kn+1.$ The numbers $b^{0}_0,\dots, b^{0}_s$ are equal to the horizontal steps in the bounce path. Here we recall Loehr's definition and check that it matches with the path from $a_0$ to $a_{kn}$ in the tree $T_{\Delta}.$ We also generalize the bounce path and statistic to the case $m=kn-1$ by considering the path from $a_0$ to $a_{\infty}.$

\begin{definition}(\cite{loehr})
Let $D$ be a Young diagram contained 
below the diagonal in an $m\times n$-rectangle. Let $m=kn+1.$ The bounce path is defined as follows. We start from the northwest corner. We alternate southward and eastward steps with the first step going southward. We always stay outside the diagram $D.$ 

On each southward step we go south until we hit a horizontal piece of the boundary of $D.$ Each eastward step is equal to the sum of the last $k$ southward steps (if there were less than $k$ southward steps yet, then it is equal to the sum of all preceding southward steps).   
\end{definition}

Let us introduce the coordinates so that the southwest corner is $(0,0).$ Then the bounce path starts from $(0,n)$ and finishes at $(kn,0).$ The bounce statistic is defined as follows:

\begin{definition}\cite{loehr}
Let $(v_0,\dots,v_a)$ and $(h_0,\dots, h_a)$ be the vertical and horizontal steps of the bounce path correspondingly. Then the statistic $\bounce(D)$ is defined by the formula:
$$
\bounce(D):=(n-v_0)+(n-v_0-v_1)+\dots+(n-\sum\limits_{0\le i\le a} v_i).
$$
In other words, $\bounce(D)$ is equal to the sum of vertical coordinates of the southwest corners of the bounce path.
\end{definition}

Let now $\Delta$ be a semimodule over $\Gamma_{m,n},$\ $m=kn+1.$ Let $D$ be the corresponding Young diagram, and, as in the Section \ref{section G_m G_n}, $G_{kn+1}(D)$ be the diagram with columns $g(a_0),\dots,g(a_{kn}).$ 

Let $T_{\Delta}$ be the bounce tree, and $a_0\rightarrow a_{b^{0}_0}\rightarrow\dots\rightarrow a_{b^{0}_0+\dots+b^{0}_s}=a_{kn}$ be the path from $a_0$ to $a_{kn}$ in it.

\begin{theorem}\label{tree to path}
The horizontal steps $(h_0, \dots, h_a)$ of the bounce path for $G_{kn+1} ( D )$ are equal to $(b^{0}_0, 
\dots,b^{0}_s).$
In particular, $a=s$.
\end{theorem}

\begin{proof}
Recall the formulae (\ref{BtoC}) and (\ref{CtoB}) (we plug $l=0$):
$$
b^{0}_i=
\left\{
\begin{array}{l}
\sum\limits_{j=i-k}^{j=i-1} c^{0}_j, i\neq k-1 \\
1+\sum\limits_{j=i-k}^{j=i-1} c^{0}_j,i=k-1.
\end{array}
\right.
$$
\noindent and
$$
c^{0}_i=g(a_{b^{0}_0+b^{0}_1+\dots+b^{0}_{i-1}})-g(a_{b^{0}_0+b^{0}_1+\dots+b^{0}_i}),\quad c^{0}_0=g(a_0)-g(a_{b^{0}_0}).
$$
\noindent where $c^{0}_i$ is the number of $n$-cogenerators in the interval $(in,(i+1)n].$

We immediately see that
$$
c^{0}_i=
\left\{
\begin{array}{l}
0, i<-2\\
1, i=-2\\
n-g(a_0)-1, i=-1
\end{array}
\right.
$$
Indeed, $-n$ is the smallest $n$-cogenerator, and there are exactly $n-g(a_0)$\ $n$-cogenerators less than zero.

For $i\ge 0$ the recurrence relations on the numbers $c^{0}_i$ and $b^{0}_i$ are almost the same as the definitions of the vertical and horizontal steps of the bounce path correspondingly. There are two differences:

\begin{enumerate}
\item The first vertical step equals  $v_0=n-g(a_0)=c^{0}_{-1}+1=c^{0}_{-1}+c^{0}_{-2}.$
\item For $i=k-1$ one has $b^{0}_{k-1}=1+\sum\limits_{j=-1}^{j=k-2} c^{0}_j=\sum\limits_{j=-2}^{j=k-2} c^{0}_j.$
\end{enumerate}

One immediately sees that those differences cancel each other. Therefore one gets
$$
v_i=
\left\{
\begin{array}{l}
c^{0}_{i-1}, i>0\\
c^{0}_{-1}+c^{0}_{-2}, i=0
\end{array}
\right.
$$
and
$$
h_i=b^{0}_i
$$
\noindent for all $i\ge 0.$
\end{proof}

\begin{theorem}
\label{BounceToArea}
We get the following relation:
$$
\bounce(G_{kn+1}(D))=\delta-|D|.
$$
\end{theorem}

\begin{proof}
Indeed, we have 
$$
\delta-|D|=\sharp(\mathbb Z_{> 0}\backslash\Delta)=\sum\limits_{i=0}^{\infty}\sharp\left((in,(i+1)n\ ]\backslash\Delta\right).
$$
By Lemma \ref{paths}, $[a_{b^{0}_0+\dots+b^{0}_i},ni]\subset\Delta.$ We get
$$
\sharp(\ (in,(i+1)n\ ]\backslash\Delta)=g(a_{b^{0}_0+\dots+b^{0}_{i-1}}) \ \mbox{and}\ \sharp(\ (0,n]\backslash\Delta)=g(a_0).
$$
Therefore,
$$
\delta-|D|=g(a_0)+g(a_{b^{0}_0})+g(a_{b^{0}_0+b^{0}_1})+\dots+g(a_{kn}),
$$
Note that this matches the definition of the bounce statistic.
\end{proof}

\begin{remark}
It follows from the proof of Theorem \ref{BounceToArea} that the number $\sharp([in,(i+1)n]\setminus \Delta)$ is equal to the vertical coordinate of the $i$-th southwest corner of the bounce path.
\end{remark}

We illustrate the bounce path in Figure \ref{bounce path}.

\begin{figure}
\centering
\begin{tikzpicture}
\draw [dashed] (0,0)--(0,3)--(7,3)--(7,0)--(0,0);
\draw [dashed] (0,1)--(7,1);
\draw [dashed] (0,2)--(7,2);
\draw [dashed] (1,1)--(1,3);
\draw [dashed] (2,1)--(2,3);
\draw [dashed] (3,1)--(3,3);
\draw [dashed] (4,0)--(4,3);
\draw [dashed] (5,0)--(5,3);
\draw [dashed] (6,0)--(6,3);
\draw (0,0)--(0,1)--(3,1)--(3,0)--(0,0);
\draw [line width=0.5mm] (0,3)--(0,1)--(4,1)--(4,0)--(6,0);
\draw (1,0.5) node {$G_7(\Delta)$};
\draw (2.05,0.82) node {$\circlearrowleft$};
\draw (5.05,-0.18) node {$\circlearrowleft$};
\end{tikzpicture}
\caption{The diagram $G_{7}(\Delta)$ with the bounce path for the $\Gamma_{3,7}$-semimodule $\Delta=\{0,2,3,5,6,7,8,9,\ldots\}.$ Here the $7$-generators are $0,2,3,5,6,8,11,$ $g(0)=g(2)=g(3)=1,$ and $g(5)=g(6)=g(8)=g(11)=0.$}
\label{bounce path}
\end{figure}
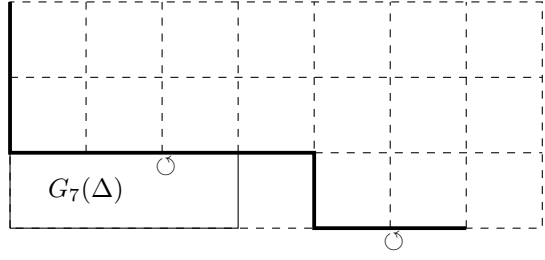

We use the above observations to generalize the bounce path and statistic to the case $m=kn-1.$

Let $\Delta$ be a $\Gamma_{m,n}$--semimodule, $m=kn-1,\  D=D(\Delta).$ Consider the bounce tree $T_{\Delta}.$ Let $a_0\rightarrow a_{b^{0}_0}\rightarrow\dots\rightarrow a_{b^{0}_0+\dots+b^{0}_s}\rightarrow a_{\infty}$ be the path from $a_0=0$ to $a_{\infty}$ in the tree $T_{\Delta}.$ Consider the Young diagram $G_{kn-1}(D)$ embedded in the $m\times n$ rectangle below the diagonal. 

\begin{definition}
The bounce path for $G_{kn-1}(D)$ starts at the northwest corner. It consists of alternating southward and eastward steps, starting with a southward step. 

On each southward step we go south until we hit a horizontal piece of the boundary of $G_{kn-1}(D).$ The $i$th eastward step equals  $b^{0}_i,$ starting with $i=0.$
\end{definition}  

The bounce statistic is defined in the same way as in the $m=kn+1$ case:

\begin{definition}
Let $(v_0,\dots,v_s)$ and $(h_0,\dots, h_s)$ be the vertical and horizontal steps of the bounce path correspondingly. Then the $\bounce(G_{kn-1}(D))$ is given by
$$
\bounce(G_{kn-1}(D))=(n-v_0)+(n-v_0-v_1)+\dots+(n-\sum\limits_{0\le i\le s} v_i).
$$
In other words, $\bounce(G_{kn-1}(D))$ is equal to the sum of vertical coordinates of the southwest corners of the bounce path.
\end{definition}

The following formula relating the bounce of $G_{kn-1}(D)$ with the area of $D$ is proved in the same way as in the $m=kn+1$ case:

\begin{theorem}
We get the following relation:
$$
\bounce(G_{kn-1}(D))=\delta-|D|.
$$
\end{theorem}


Finally, we expand the recurrence relations involved in the definition of the bounce path in this case to get a simple description of the bounce path in terms of the diagram $G_{kn-1}(D).$ This is parallel with the proof of the Theorem \ref{tree to path} in the $m=kn+1$ case.

Recall the formulae (\ref{BtoC-1}) and (\ref{CtoB-1}) (we plug $l=0$):
$$
b^{0}_i=
\left\{
\begin{array}{l}
\sum\limits_{j=i-k}^{j=i-1} c^{0}_j, i\neq k-1 \\
-1+\sum\limits_{j=i-k}^{j=i-1} c^{0}_j,i=k-1.
\end{array}
\right.
$$

And
$$
c^{0}_i=g(a_{b^{0}_0+b^{0}_1+\dots+b^{0}_{i-1}})-g(a_{b^{0}_0+b^{0}_1+\dots+b^{0}_i}),\quad c^{0}_0=g(a_0)-g(a_{b^{0}_0}),
$$
\noindent where $c^{0}_i$ is the number of $n$-cogenerators in the interval $[in,(i+1)n).$

We immediately see that
$$
c^{0}_i=
\left\{
\begin{array}{l}
0, i<-1\\
n-g(a_0), i=-1
\end{array}
\right.
$$
Indeed, $-n$ is the smallest $n$-cogenerator, and there are exactly $n-g(a_0)$\ $n$-cogenerators less than zero.

Comparing the recurrence relations one immediately gets
$$
v_i=c^{0}_{i-1}\quad \mbox{\rm and}\quad
h_i=b^{0}_i
\quad \mbox{\rm for all}\quad  i\ge 0.$$

Therefore, the bounce path in the $m=kn-1$ case is constructed in the same way as in the $m=kn+1$ case, except that the $(k-1)$-th horizontal step is shorter by $1$ (we count steps starting from $0$):
$$
h_i=
\left\{
\begin{array}{l}
\sum\limits_{j=i-k}^{j=i-1} v_j, i\neq k-1 \\
-1+\sum\limits_{j=i-k}^{j=i-1} v_j,i=k-1.
\end{array}
\right.
$$

We illustrate the bounce path in Figure \ref{bounce path-1}.

Let us return to the (3,8)-module $\Delta=\{0,3,4,6,7,8,9,\ldots\}$ from example in Figure \ref{tree-1}.
The $8$-generators are $0,3,4,6,7,9,10,13$, so the diagram $G_{8}(\Delta)$ with the bounce path is shown in Figure \ref{bounce path-1}.

\begin{figure}
\centering
\begin{tikzpicture}
\draw [dashed] (0,0)--(0,3)--(8,3)--(8,0)--(0,0);
\draw [dashed] (1,1)--(8,1);
\draw [dashed] (0,2)--(8,2);
\draw [dashed] (1,1)--(1,3);
\draw [dashed] (2,1)--(2,3);
\draw [dashed] (3,1)--(3,3);
\draw [dashed] (4,0)--(4,3);
\draw [dashed] (5,0)--(5,3);
\draw [dashed] (6,0)--(6,3);
\draw [dashed] (7,0)--(7,3);
\draw (0,0)--(0,2)--(1,2)--(1,1)--(3,1)--(3,0)--(0,0);
\draw [line width=0.5mm] (0,3)--(0,2)--(1,2)--(1,1)--(3,1)--(3,0)--(8,0);
\draw (1,0.5) node {$G_8(\Delta)$};
\draw (5.05,-0.18) node {$\circlearrowleft$};
\draw (7.05,-0.18) node {$\circlearrowleft$};
\end{tikzpicture}
\caption{The diagram $G_{8}(\Delta)$ with the bounce path for the $\Gamma_{3,8}$-semimodule $\Delta=\{0,3,4,6,7,8,9,\ldots\}.$ Here the $8$-generators are $0,3,4,6,7,9,10,13,$ $g(0)=2,$ $g(3)=g(4)=1,$ and $g(6)=g(7)=g(9)=g(10)=g(13)=0.$}
\label{bounce path-1}
\end{figure}
\section{Bijective proof of symmetry for $m\le 3$}\label{Section m<4}

In this section we give a bijective proof of the identity
$c_{m,n}(q,t)=c_{m,n}(t,q)$ for $n=2$ and $n=3$.

The $(2,2k+1)$ case is very simple. We should consider Young diagrams in $2\times (2k+1)$ rectangle below the diagonal. Such a diagram $D_i$ can have only one row with $i\le k$ boxes in it, and $|D_i|=h_{+}(D_i)=i$. Therefore, the polynomial $c_{2,2k+1}(q,t)$ is given by the following formula:

$$
c_{2,2k+1}(q,t)=q^k+q^{k-1}t+\dots+qt^{k-1}+t^k,
$$
\noindent which is obviously symmetric in $q$ an $t.$ The corresponding involution on diagrams sends the diagram $D_i$ to $D_{k-i}.$

The case $(3,n)$ turns out to be more subtle. We consider Young diagrams in $3\times n$ rectangle below the diagonal. Such a diagram $D_{\alpha,\beta}$ has two rows of length $\alpha$ and $\beta$, $\alpha\le \beta$. Moreover, 

$$
\alpha\le k:=\lfloor\frac{m}{3}\rfloor,\quad \beta\le \lfloor\frac{2m}{3}\rfloor.
$$ 

Consider the map $\phi:D\mapsto (\delta-|D|,h_+(D))$ from the set of diagrams in $3\times m$ rectangle below the diagonal to $\mathbb Z^2.$

\begin{theorem}
The map $\phi$ is injective. The image of $\phi$ is the set of integer points inside the triangle:

$$
\Td:=\left\{(a,b): a+b\le \delta, a+2b\ge \delta, 2a+b\ge \delta\right\}.
$$
\end{theorem}

\begin{proof}

The area of $D_{\alpha,\beta}$ equals $\alpha+\beta$. To compute $h_{+}(D_{\alpha,\beta})$, we consider three cases:

\begin{enumerate}
\item $0\le\alpha\le\beta\le k$. In this case one can check that $h_{+}(D_{\alpha,\beta})=\beta$. Therefore, $\phi(D_{\alpha, \beta})=(\delta-\alpha-\beta,\beta),$ and the image of $\phi$ is given by the formula
$$
\{(a,b)\in \Td\cap\mathbb Z^2: b\le k\}.
$$

\item $\beta>k, \beta-\alpha\le k$. In this case $h_{+}(D_{\alpha,\beta})=2\beta-k, \quad \phi(D_{\alpha, \beta})=(\delta-\alpha-\beta,2\beta-k),$ and the image of $\phi$ is given by the formula
$$
\{(a,b)\in \Td\cap\mathbb Z^2: b> k,\ b+k\ \mbox{\rm even}\}.
$$ 

Indeed, the boxes contributing to $h_{+}(D_{\alpha,\beta})$ form three groups of sizes
$\alpha, \beta-\alpha,$ and  $\beta-k$, hence
$$h_{+}(D_{\alpha,\beta})=(\beta-\alpha)+\alpha+(\beta-k)=2\beta-k.$$
We show these boxes in Figure \ref{bounceboxes}.

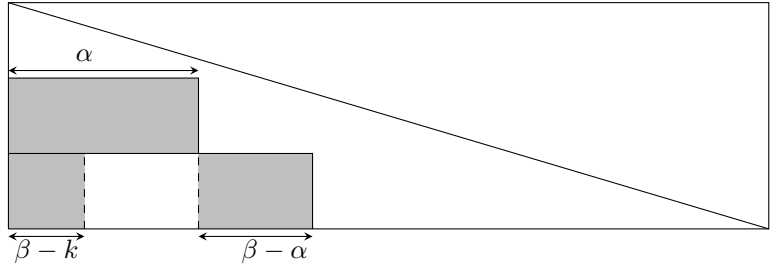
\begin{figure}
\begin{tikzpicture}
\fill [color=lightgray] (0,0)--(0,1)--(1,1)--(1,0);
\fill [color=lightgray] (0,1)--(0,2)--(2.5,2)--(2.5,1);
\fill [color=lightgray] (2.5,0)--(2.5,1)--(4,1)--(4,0);

\draw (0,0)--(0,3)--(10,3)--(10,0)--(0,0);
\draw (0,3)--(10,0);
\draw (0,2)--(2.5,2)--(2.5,1)--(4,1)--(4,0);
\draw (0,1)--(2.5,1);
\draw [dashed] (1,0)--(1,1);
\draw [dashed] (2.5,0)--(2.5,1);
\draw [<->,>=stealth] (0,-0.1)--(1,-0.1);
\draw [<->,>=stealth] (2.5,-0.1)--(4,-0.1);
\draw [<->,>=stealth] (0,2.1)--(2.5,2.1);
\draw (1,2.3) node {$\alpha$};
\draw (3.5,-0.3) node {$\beta-\alpha$};
\draw (0.5,-0.3) node {$\beta-k$};
\end{tikzpicture}
\caption{Boxes contributing to $h_{+}(D_{\alpha,\beta})$ for $\beta>k, \beta-\alpha\le k$.}
\label{bounceboxes}
\end{figure}

\item $\beta-\alpha> k$. In this case $h_{+}(D_{\alpha,\beta})=2\alpha+k+1,\quad \phi(D_{\alpha, \beta})=(\delta-\alpha-\beta,2\alpha+k+1),$ and the image of $\phi$ is given by the formula
$$
\{(a,b)\in \Td\cap\mathbb Z^2: b> k,\ b+k\ \mbox{\rm odd}\}. 
$$ 
\end{enumerate}
Note that these three image sets cover all integer points in $\Td$ with no overlaps.
\end{proof}
Note that the triangle $\Td$ is symmetric with respect to the diagonal $a=b.$ Therefore, one gets the following
\begin{corollary}
There exist an involution ${\bf i}$ on the set of diagrams in $3\times n$ rectangle below the diagonal, such that $\delta-|D|=h_+({\bf i}(D)).$ In particular, the polynomial $c_{3,n}(q,t)$ is symmetric.
\end{corollary}

\begin{proof}
The involution ${\bf i}$ is given by the formula

$$
{\bf i}=\phi^{-1}\circ {\bf s} \circ \phi,
$$

\noindent where ${\bf s}:\mathbb R^2\to \mathbb R^2$ is the symmetry with respect to the diagonal $a=b.$
\end{proof}

{\bf Acknowledgements.}

The authors are grateful to D. Armstrong, L. G\"ottsche, J. Haglund, A. Iarrobino, L. Migliorini, A. Oblomkov, J. Rasmussen and V. Shende
for  useful discussions. Special thanks to K. Lee, L. Li and N. Loehr for communicating us their proof of Conjecture \ref{SymConj} for $m\le 4$.The research of E. G. was partially supported by the grants RFBR-10-01-00678, NSh-8462.2010.1 and the Simons foundation.


\begin{thebibliography}{99}

\bibitem{AK} 
A. Altman and S. Kleiman, Compactifying the Picard Scheme.
Adv. in Math. {\bf 35} (1980), 50-112.

\bibitem{AIK}
A. Altman, A. Iarrobino, and S. Kleiman, Irreducibility of the 
Compactified Jacobian. Nordic Summer School NAVF, Oslo 1976, Noordhoff (1977).

\bibitem{anderson} J. Anderson. Partitions which are simultaneously $t_1$- and $t_2$-core. Discrete Math. {\bf 248} (1-3) (2002).  
237--243.

\bibitem{arms} D. Armstrong. Rational Catalan Combinatorics, slides from a talk at JMM 2012, Boston.

\bibitem{beauville} A. Beauville. Counting rational curves on K3--surfaces. Duke Math. J. {\bf 97}
(1999), 99--108.

\bibitem{delgado} F. Delgado de la Mata. Gorenstein curves and symmetry of the semigroup of values. Manuscripta Math. {\bf 61} (1988), no. 3, 285--296.

\bibitem{FGvS} B. Fantechi, L. G\"ottsche, and D. van Straten, 
Euler number of the compactified Jacobian and multiplicity 
of rational curves. J. Alg. Geom. {\bf 8} (1999), no. 1, 115--133. 

\bibitem{FMS} B. Ford, H. Mai, L. Sze. Self-conjugate $p$- and $q$-core partitions and blocks of $A_n$.
Journal of Number Theory {\bf 129} (2009), 858--865.

\bibitem{GH} A. Garsia, M. Haiman. A remarkable $q,t$-Catalan sequence and $q$-Lagrange inversion. J. Algebraic Combin. {\bf 5} (1996), no. 3, 191--244.

\bibitem{gahagl} A. Garsia, J. Haglund. A proof of the $q,t$--Catalan positivity conjecture.
Discrete Math. {\bf 256} (2002), no. 3, 677--717.


\bibitem{GM} E. Gorsky, M. Mazin. Compactified Jacobians and q,t-Catalan Numbers. Journal of Combinatorial Theory, Series A {\bf 120} (2013), pp. 49-63. 

\bibitem{GORS} E. Gorsky, A. Oblomkov, J. Rasmussen, V. Shende.Torus knots and the rational DAHA. arXiv:1207.4523

\bibitem{GP} G.--M. Greuel, G. Pfister. Moduli Spaces for Torsion Free Modules on
Curve Singularities I. J. Alg. Geom. {\bf 2} (1993), 81--135.

\bibitem{hagl1} J. Haglund. Conjectured statistics for the $q,t$--Catalan numbers. Adv. Math. {\bf 175} (2003), no. 2, 319--334.

\bibitem{hagl} J. Haglund. The $q,t$-Catalan numbers and the space of diagonal harmonics. With an appendix on the combinatorics of Macdonald polynomials. University Lecture Series, 41. American Mathematical Society, Providence, RI, 2008.



\bibitem{hikita} T. Hikita. Affine Springer fibers of type $A$ and combinatorics of diagonal coinvariants. arXiv:1203.5878

\bibitem{kunz} E. Kunz. The value-semigroup of a one-dimensional Gorenstein ring. Proc. Amer. Math. Soc. {\bf 25} (1970) 748--751. 

\bibitem{lll} K. Lee, L. Li, N. Loehr. Combinatorics of certain higher $q,t$-Catalan polynomials: chains, joint symmetry, and the Garsia-Haiman formula. arXiv:1211.2191


\bibitem{loehr} N. Loehr. Conjectured Statistics for the Higher $q,t$-Catalan Sequences. The Electronic Journal of Combinatorics {\bf 12} (2005), Research Paper 9.


\bibitem{lowa} N. Loehr, G. Warrington. A continuous family of partition statistics equidistributed
with length.  Journal of Combinatorial Theory, Series A {\bf 116} (2009), 379--403.

\bibitem{LS} G. Lusztig, J. M. Smelt. Fixed point varieties on the space of lattices.
Bull. London Math. Soc. {\bf 23} (1991), no. 3, 213--218.




\bibitem{mist} J. W. Milnor, J. D. Stasheff. Characteristic classes. Annals of mathematics studies, 76.
Princeton, N.J., Princeton University Press, 1974. 


 
\bibitem{OS} A. Oblomkov, V. Shende. The Hilbert scheme 
of a plane curve singularity and the HOMFLY polynomial of its link. 
Duke Math. J. {\bf 161} (2012), no.7,  1277--1303 

\bibitem{ORS} A. Oblomkov, J. Rasmussen, V. Shende. The Hilbert scheme of a plane curve singularity and the HOMFLY homology of its link 
. arXiv:1201.2115.

 
\bibitem{piont}  J. Piontkowski.  Topology of the compactified Jacobians
of singular curves, Math. Z. {\bf 255} (2007) , no. 1, 195--226.



\bibitem{puchta} J. C. Puchta. Partitions which are $p$- and $q$-core. Integers {\bf 1} (2001), A6, 3 pp.

\end{thebibliography}
\end{document}